\documentclass[12pt,x11names]{article}
\usepackage{geometry,amsmath,amssymb,amsthm,graphicx,epsfig}
\usepackage{subcaption,enumerate,dcolumn,latexsym,epstopdf,color}
\usepackage[graph]{xy}
\usepackage[colorlinks=true,linkcolor=blue,citecolor=blue]{hyperref}
\usepackage[usenames,dvipsnames,svgnames,table]{xcolor}
\usepackage{ulem}
\usepackage{hyperref}
\numberwithin{equation}{section}
\usepackage{multirow}
\usepackage{tabularx}

\usepackage{pdfsync,comment,a4wide}
\usepackage{framed}
\usepackage{enumitem}
\usepackage[titletoc,title]{appendix}
\newtheorem{theorem}{Theorem} 
\newtheorem{proposition}{Proposition} 
\newtheorem{lemma}{Lemma} 
\theoremstyle{definition}
\newtheorem{remark}{Remark} 
\newtheorem{assumption}{Assumption}
\renewcommand{\theassumption}{\Alph{assumption}}

\renewcommand{\em}[1]{\normalem \em{#1}}
\renewcommand{\emph}[1]{\normalem \emph{#1}}
\newcommand{\ie}{{\em i.e.}, }

\newcommand{\cf}{{\em cf.\ }}

\newcommand{\nn}{\mathbb{N}} 
\newcommand{\real}{\mathbb{R}} 
\newcommand{\Fix}{\mathrm{Fix}} 
\newcommand{\norm}[1]{\left\Vert {#1} \right\Vert} 

\DeclareMathOperator*{\argmin}{\arg\!\min}
\newcommand{\prox}{\mathrm{prox}} 
\newcommand{\act}[1]{\left\langle {#1} \right\rangle} 
\newcommand{\seq}[2]{\left\{{#1}_{{#2}}\right\}_{{#2} \in \mathbb{N}}}
\newcommand{\Seq}[2]{\left\{{#1}^{{#2}}\right\}_{{#2} \in \mathbb{N}}}

\newcommand{\bo}{{\bf 0}}

\newcommand{\bu}{{\bf u}}

\newcommand{\bx}{{\bf x}}
\newcommand{\by}{{\bf y}}
\newcommand{\bz}{{\bf z}}
\newcommand{\bb}{{\bf b}}

\newcommand{\bba}{{\bf A}}

\newcommand{\bbq}{{\bf Q}}
\newcommand{\bbl}{{\bf L}}
\newcommand{\bbi}{{\bf I}}

\title{A First Order Method for Solving Convex Bi-Level Optimization Problems}
\author{Shoham Sabach\footnote{Department of Industrial Engineering and Management, Technion---Israel Institute of Technology, Haifa 3200003, Israel. E-mail: ssabach@ie.technion.ac.il.} \and Shimrit Shtern\footnote{Department of Industrial Engineering and Management, Technion---Israel Institute of Technology, Haifa 3200003, Israel. E-mail: shimrits@tx.technion.ac.il.}}
\date{\today}

\begin{document}
\maketitle

\begin{abstract}
	In this paper we study convex bi-level optimization problems for which the inner level consists of 
	minimization of the sum of smooth and nonsmooth functions. The outer level aims at minimizing a 
	smooth and strongly convex function over the optimal solutions set of the inner problem. We 
	analyze a first order method which is based on an existing fixed-point algorithm. Global sublinear 
	rate of convergence of the method is established in terms of the inner objective function values. 		
	
\end{abstract}
\section{Introduction} \label{Sec:Introduction}
	In this paper we are interested in the following bi-level optimization problem (where we use the 
	terminology of inner and outer levels). The \emph{outer} level is given by the following 
	constraint minimization problem
 	\begin{equation} \label{Prob:MNP} \tag{MNP}
		\min_{\bx \in X^{\ast}}  \omega\left(\bx\right),
 	\end{equation}
	where $\omega$ is a strongly convex and differentiable function while $X^{\ast}$ is the, assumed 
	nonempty, set of minimizers of the \emph{inner} level problem, which is the classical convex 
	composite model, given by
	\begin{equation} \label{Prob:P} \tag{P}
		\min_{\bx \in \real^{n}} \left\{ \varphi\left(\bx\right) := f\left(\bx\right) + g\left(\bx
		\right) \right\},
	\end{equation}
	where $f$ is a continuously differentiable function and $g$ is an extended valued (possibly 
	nonsmooth) function, see next section for precise assumptions. We denote the unique optimal 
	solution of problem \eqref{Prob:MNP} by $\bx_{mn}^{\ast}$, following the notation used in
	\cite{BS2014}.
\medskip

	The most known indirect method (the meaning of direct and indirect method will be made precise in 
	the following lines) for solving problem \eqref{Prob:MNP} is by the well-known Tikhonov 
	regularization \cite{TA77} which suggests solving the following alternative regularized problem, 
	for some $\lambda > 0$,
	\begin{equation} \label{Prob:regP} \tag{Q$_\lambda$}
		\min_{\bx \in \real^{n}}  \left\{ \varphi_{\lambda}\left(\bx\right) := \varphi\left(\bx\right) 
		+ \lambda\omega\left(\bx\right) \right\}.
	\end{equation}
	In \cite{FM1988} the authors treat the case that $g$ is an indicator function of a closed and 
	convex set $X$, and show that under some restrictive conditions including $X$ being a polyhedron, 
	there exists a small enough $\lambda^{\ast} > 0$ such that the optimal solution of problem 
	(Q$_{\lambda^{\ast}}$) is the optimal solution of problem \eqref{Prob:MNP}, see 
	\cite[Theorem 9]{FM1988}. However, in practice, even for this specific case, the value of 
	$\lambda^{\ast}$ is unknown, and so \eqref{Prob:regP} must be solved for a sequence of 
	regularizing parameters $\seq{\lambda}{k}$ for which $\lambda_{k} \rightarrow 0$ as $k \rightarrow 
	\infty$. In \cite{S07} Solodov showed that, provided that $\sum_{k = 1}^{\infty} \lambda_{k} =
	\infty$ and $g$ is again an indicator function of a closed and convex set, there is no need to 
	find the optimal solution of problem (Q$_{\lambda_{k}}$), $k \in \nn$, and it is sufficient to 
	approximate its solution by performing single projected gradient step on $\varphi_{\lambda_{k}}$ 
	for all $k \in \nn$. In the case that both $f$ and $\omega$ are differentiable with Lipschitz 
	continuous gradients, the generated sequence converges to the optimal solution of problem 
	\eqref{Prob:MNP}, even if $\omega$ is not strongly convex. Thus, the algorithm suggested in 
	\cite{S07} provides a \emph{direct} method for solving problem \eqref{Prob:MNP}. Another direct 
	approach to solve problem \eqref{Prob:MNP} is the \emph{Hybrid Steepest Descent Method} (HSDM) 
	presented in \cite[Section 17.3.2]{YYY11}, which was proved to converge to the optimal solution of 
	problem \eqref{Prob:MNP} provided that $\lambda_{k} \rightarrow 0$ as $k \rightarrow \infty$ and 
	$\sum_{k = 1}^{\infty} \lambda_{k} = \infty$. In \cite{NP11}, an extension of the HSDM is 
	suggested for the case where $g\left(\cdot\right) := 0$ and $\omega$ is not necessarily 
	differentiable or strongly convex but has bounded subgradients on the optimal set.
\medskip
	
	The major missing part of these papers is that while convergence was proven the convergence rates 
	of these algorithms are unknown. Very recently, a new direct first order method for solving 
	problem \eqref{Prob:MNP}, called the Minimal Norm Gradient (MNG) was proposed in \cite{BS2014}, 
	for which the authors proved an $O(1/\sqrt{k})$ rate of convergence result, in terms of the inner 
	objective function values. Even though the authors of \cite{BS2014} deal with the specific case of 
	problem \eqref{Prob:P} for which the nonsmooth function $g$ is assumed to be an indicator function 
	of a nonempty, closed and convex set, it seems that their analysis carries over even in the more 
	general setting of this paper, that is, for any convex and extended valued function $g$. The MNG 
	method is based on the cutting plane idea which means that at each iteration of the algorithm two 
	specific half-spaces are constructed and then a minimization of the outer objective function 
	$\omega$ over the intersection of these half-spaces is solved (see more detail in Section 
	\ref{SSec:MNG}). In some cases, computing a solution to this minimization task can be done 
	analytically. However, for some choice of outer function $\omega$ obtaining a solution might be 
	computationally expensive and require an additional nested algorithm to approximate its solution.
\medskip

	Inspired by \cite{BS2014} and motivated by the limitations of the MNG method (as discussed in 
	Section \ref{SSec:MNG}) we are interested in pursuing the research on bi-level optimization 
	problems in the following way. We study and analyze a first order method\footnote{which is based 
	on an existing algorithm, proposed in \cite{Xu2004}, for solving a certain class of fixed point 
	problems (see precise details in Section \ref{SSec:XuMethod}).} for solving problem 
	\eqref{Prob:MNP} with a non-asymptotic $O(1/k)$ global rate of convergence in terms of the inner 
	objective function values, which we call BiG-SAM. In addition to the improved rate of convergence, 
	BiG-SAM seems to be simpler and cheaper than the MNG method in the following sense. The operation 
	in the algorithm which relates to the inner problem is of the same complexity as in the MNG 
	method. On the other hand, the operation with respect to the outer problem is very simple in our 
	case, and consists of computing the gradient of the objective function $\omega$, in the MNG method 
	two tasks are needed: computation of the gradient of $\omega$ and a minimization of $\omega$ over 
	the intersection of two half-spaces, whose computational cost highly depend on the function 
	$\omega$, as discussed below (see Section \ref{SSec:XuMethod}).
\medskip

	Another contribution of this paper is the fact that BiG-SAM can also be used in situations for 
	which the outer objective function $\omega$ is strongly convex but not necessarily smooth. In this 
	case, we show that BiG-SAM solves problem \eqref{Prob:MNP} but with respect to the Moreau envelope 
	of $\omega$ instead of $\omega$ itself. In this case we offer a new concept of measuring rate of 
	convergence. This property of BiG-SAM allows considering outer functions which are not necessarily 
	smooth and include, for example, functions with sparsity terms (see more details in Section 
	\ref{SSec:SAMNonsmooth}).
\medskip

	The paper is organized in the following way. In Section \ref{Sec:AlgFrameMathTools} we discuss the 
	optimization framework of the class of bi-level problems, then we give all notations and auxiliary 
	results that are needed for the forthcoming sections. We conclude this section with a short 
	overview of the MNG method developed in \cite{BS2014} (see Section \ref{SSec:MNG}). Section 
	\ref{Sec:SAM} is devoted to an algorithm which forms the basis of BiG-SAM. We first discuss it in 
	its generality for solving a certain class of fixed-point problems (see Section 
	\ref{SSec:XuMethod}) and then we specify it for solving the bi-level problems described in Section 
	\ref{SSec:BiLevelOpt}. In Section \ref{Sec:ConvergenceAnalysis} we prove rate of convergence 
	results of BiG-SAM. This section also includes our results in the case where the outer function is 
	not necessarily smooth. Section \ref{Sec:Numerical Experiments} contains numerical experiments 
	comparing MNG to BiG-SAM and showing its computational superiority in obtaining faster rates.
\medskip

	Throughout the paper we denote vectors by boldface letters. The notation $\act{\cdot , \cdot}$ is 
	used to denote the inner product of two vectors and $\norm{\cdot}$ is the norm associated with 
	this inner product, unless stated otherwise.

\section{Optimization Framework and Mathematical Tools} \label{Sec:AlgFrameMathTools}
\subsection{Convex Bi-Level Optimization} \label{SSec:BiLevelOpt}
	In this paper we are focusing on bi-level optimization problems which are formulated as follows. 
	We first discuss the inner level problem which is given by,
	\begin{equation} \label{Prob:PP} \tag{P}
		\min_{\bx \in \real^{n}} \left\{ \varphi\left(\bx\right) := f\left(\bx\right) + g\left(\bx
		\right) \right\}.
	\end{equation}
	The standing assumption on the functions of the inner level problem is recorded now.
	\begin{assumption} \label{A:Composite}
		\begin{itemize}
			\item[$\rm{(i)}$] $f : \real^{n} \rightarrow \real$ is convex and continuously 
				differentiable such that its gradient is Lipschitz with constant $L_{f}$, that is,
				\begin{equation*}
					\norm{\nabla f\left(\bx\right) - \nabla f\left(\by\right)} \leq L_{f}\norm{\bx -
					\by}, \quad \forall \,\, \bx , \by \in \real^{n}.
				\end{equation*}
			\item[$\rm{(ii)}$] $g : \real^{n} \rightarrow \left(-\infty , \infty\right]$ is proper, 
				lower semicontinuous and convex.
			\item[$\rm{(iii)}$] The set $X^{\ast}$ of all optimal solutions of problem \eqref{Prob:P} 
				is nonempty, that is, $X^{\ast} \neq \emptyset$.
		\end{itemize}
	\end{assumption}
	Problem \eqref{Prob:P}, which consists of minimizing the sum of a smooth function $f$ and a 
	possibly nonsmooth function $g$, is one of the most studied models in modern optimization with a 
	huge body of literature (see, for instance, \cite{BT10} and the references therein). The basic 
	algorithm for solving problem \eqref{Prob:P} is the so called \emph{Proximal Gradient} (PG) or 
	proximal forward-backward method (see \cite{B77,P79} for the origin of the algorithm and 
	\cite{BT09} for the rate of convergence result including an accelerated version), which 
	iteratively generates a sequence $\Seq{\bx}{k}$ starting from an arbitrary point $\bx^{0} \in 
	\real^{n}$ via the following rule
	\begin{equation}
		\bx^{k + 1} = \prox_{tg}\left(\bx^{k} - t\nabla f\left(\bx^{k}\right)\right), \quad k \in \nn,
	\end{equation}
	for some step-size $t > 0$. The main operation of this algorithm is the computation of the 
	\textit{Moreau proximal mapping} of a proper, lower semicontinuous and convex function $h : 
	\real^{n} \rightarrow \left(-\infty , \infty\right]$ which is denoted and defined by
	\begin{equation}
		\prox_{h}\left(\bx\right) := \argmin_{\bu \in \real^{n}} \left\{ h\left(\bu\right) + \frac{1}
		{2}\norm{\bu - \bx}^{2} \right\}.
	\end{equation}
	The PG method can also be seen as a fixed-point algorithm where the iterated mapping is given 
	(using the notation of \cite{BS2014}) by
	\begin{equation} \label{ProxGradMap}
		T_{t}\left(\bx\right) := \prox_{tg}\left(\bx - t\nabla f\left(\bx\right)\right),
	\end{equation}
	and is called the \emph{prox-grad mapping}. In the case where $g$ is the indicator function 
	$\delta_{X}$ of a set $X$, defined to be zero on $X$ and $+\infty$ on $\real^{n} \setminus X$, the 
	prox-grad mapping coincides with the \emph{proj-grad mapping} (which is discussed in
	\cite{BS2014,N04}), since in this case the proximal mapping of $g$ is exactly the orthogonal 
	projection onto $X$. The prox-grad mapping possesses the following two important properties which 
	are relevant to our analysis (see \cite[Proposition 12.27, Page 176]{BC2011-B} and 
	\cite[Section 2.3.2, Page 48]{BT10}). The second property characterizes the set of all fixed 
	points of $T_{t}$, which is denoted by $\Fix(T_{t})$ and defined by $\Fix(T_{t}) = \left\{ \bx \in 
	\real^{n} : \, T_{t}\left(\bx\right) = \bx \right\}$.
	\begin{lemma} \label{L:ProxGradProp}
		\begin{itemize}
			\item[$\rm{(i)}$] The prox-grad mapping $T_{t}$ is nonexpansive for all $t \in \left(0 , 
				1/L_{f}\right]$, that is,
				\begin{equation}
					\norm{T_{t}\left(\bx\right) - T_{t}\left(\by\right)} \leq \norm{\bx - \by}, 
					\quad \forall \,\, \bx , \by \in \real^{n}.
				\end{equation}
			\item[$\rm{(ii)}$] Fixed points of the prox-grad mapping $T_{t}$ are optimal solutions of 
				problem \eqref{Prob:P} and vice versa, that is, 
				\begin{equation}
					\bx \in X^{\ast} \quad \Leftrightarrow \quad \bx = T_{t}\left(\bx\right) = 
					\prox_{tg}\left(\bx - t\nabla f\left(\bx\right)\right).
				\end{equation}
				Therefore, we have that $\Fix(T_{t}) = X^{\ast}$ for all $t > 0$.
		\end{itemize}
	\end{lemma}
	The following result will be essential for the rate of convergence analysis presented in Section 
	\ref{Sec:ConvergenceAnalysis} (\cf \cite[Lemma 2.3, Page 190]{BT09}).
	\begin{proposition} \label{P:ProximalInequality}
		Suppose that Assumption \ref{A:Composite} holds true. Let $\bx \in \real^{n}$ and denote $\bx^ 
		{+} = T_{t}\left(\bx\right)$. Then, for any $ t \leq 1/L_{f}$ and $\bu \in \real^{n}$, we have
		\begin{equation}
			\varphi\left(\bx^{+}\right) - \varphi\left(\bu\right) \leq \frac{1}{t}\act{\bx - \bx^{+} ,
			\bx - \bu} - \frac{1}{2t}\norm{\bx - \bx^{+}}^{2}.
		\end{equation}
	\end{proposition}
	To conclude the discussion on the inner problem \eqref{Prob:P}, we note that it is well-known that 
	the PG method has an $O(1/k)$ rate of convergence in terms of the objective function values 
	$\varphi$ (see \cite{BT09}). In this respect, the method proposed in this paper shares the same 
	rate of convergence as the PG method while capable of solving the more complicated bi-level 
	problem described in detail now.
\medskip

	We now turn to discuss the outer problem. As mentioned in the introduction, the outer problem is 
	given by the following convex constrained problem
	\begin{equation} \label{Prob:MNPP} \tag{MNP}
		\min_{\bx \in X^{\ast}}  \omega\left(\bx\right),
	\end{equation}
	where $X^{\ast}$ is the optimal solution set of problem \eqref{Prob:P}. Here, using the same 
	terminology as in \cite{BS2014}, we refer to this outer problem as the \emph{Minimal Norm Problem} 
	(MNP). The standing assumption on the objective function $\omega$ of problem \eqref{Prob:MNP} is 
	recorded now.
	\begin{assumption} \label{A:Omega}
		\begin{itemize}
			\item[$\rm{(i)}$] $\omega : \real^{n} \rightarrow \real$ is strongly convex with parameter 
				$\sigma > 0$.
			\item[$\rm{(ii)}$] $\omega$ is a continuously differentiable function such that $\nabla 
				\omega$ is Lipschitz continuous with constant $L_{\omega}$.
		\end{itemize}
	\end{assumption}
	It should be noted that Assumption \ref{A:Omega}(i) here is slightly stronger than the 
	corresponding assumption given in \cite[Page 27]{BS2014}, since we do not only assume 
	differentiability of $\omega$, as assumed in \cite{BS2014}, but additionally assume that its 
	gradient is Lipschitz continuous. However, in practice, most of the interesting examples of 
	$\omega$ do satisfy this additional assumption. 
\medskip

	A function that would be essential in our paper is the well-known \textit{Moreau envelope} of a 
	given function $\omega$, which is denoted by $M_{s\omega}$, and defined by 
	\begin{equation}
		M_{s\omega}\left(\bx\right) = \min_{\bu \in \real^{n}} \left\{ \omega\left(\bu\right) + 
		\frac{1}{2s}\norm{\bu - \bx}^2 \right\}.
	\end{equation}
	It is well-known that $M_{s\omega}$ is continuously differentiable on $\real^{n}$ with an 
	${1}/{s}$-Lipschitz continuous gradient (see \cite[Proposition 12.29, Page 176]{BC2011-B}), which 
	is given by
	\begin{equation} \label{MoreauEnvelopeGradient}
		\nabla M_{s\omega}\left(\bx\right) = \frac{1}{s}\left(\bx - \prox_{s\omega}\left(\bx\right)
		\right).
	\end{equation}
	Another property of the Moreau envelope that plays a central role in this paper is that, if the 
	corresponding function $\omega$ is strongly convex then its Moreau envelope is also strongly 
	convex as recorded in the following result (for completeness, the proof given in Appendix
	\ref{A:ProofStrongMoreau}).
	\begin{proposition} \label{P:StrongConvexMoreauEnvelope}
		Let $\omega : \real^{n} \rightarrow \left(-\infty , \infty\right]$ be a strongly convex 
		function with strong convexity parameter $\sigma$ and let $s > 0$. Then, the Moreau envelope 
		$M_{s\omega}$ is strongly convex with parameter $\sigma/\left(1 + s\sigma\right)$.
	\end{proposition}
	See Section \ref{SSec:SAMNonsmooth} for more details on the Moreau envelope relevant for our 
	discussion.
\medskip
	
	A mapping $S : \real^{n} \rightarrow \real^{n}$ is said to be \emph{$\beta$-contraction} if there 
	exists $\beta < 1$ such that
	\begin{equation*}
		\norm{S\left(\bx\right) - S\left(\by\right)} \leq \beta\norm{\bx - \by}, \quad \forall \,\, 
		\bx , \by \in \real^{n}.
	\end{equation*}
	For functions $\omega$ which satisfy Assumption \ref{A:Omega} we have the following result which 
	is crucial for our derivations. Although this result seem to be classic, we did not find an exact 
	reference for its proof and therefore, for the sake of completeness, we provide a proof in 
	Appendix \ref{A:ProofGradContraction}.
	\begin{proposition} \label{P:GradContraction}
		Suppose that Assumption \ref{A:Omega} holds. Then, the mapping defined by $S_{s} = I - s\nabla 
		\omega$, where $I$ is the identity operator, is a contraction for all $s \leq 2/
		\left(L_{\omega} + \sigma\right)$, that is,
		\begin{equation}
			\norm{\bx - s\nabla \omega\left(\bx\right) - \left(\by - s\nabla \omega\left(\by\right)
			\right)} \leq \sqrt{1 - \frac{2s\sigma L_{\omega}}{\sigma + L_{\omega}}}\norm{\bx - \by},
			\quad \forall \,\, \bx , \by \in \real^{n}.
		\end{equation}
	\end{proposition}
	We now conclude this section by giving a short overview on the MNG method developed in
	\cite{BS2014}. 
 
\subsection{The Minimal Norm Gradient Method} \label{SSec:MNG}
	The MNG method of \cite{BS2014} was designed to tackle bi-level optimization problems for which $g
	\left(\cdot\right) := \delta_{X}\left(\cdot\right)$. In this case $\varphi\left(\bx\right) = f
	\left(\bx\right)$, for all $\bx \in X$. 
\medskip

	Each iteration of the MNG method consists of three main computational tasks.
	\begin{itemize}
		\item[$\rm{(i)}$] Computing the proj-grad mapping $T_{t}$ in order to construct the first 
			half-space. 
		\item[$\rm{(ii)}$] Computing the gradient of $\omega$, which is needed to construct the 
			second half-space.
		\item[$\rm{(iii)}$] Minimizing $\omega$ over the intersection of these two half spaces.
	\end{itemize}
	The first two tasks are standard in first order methods and consist of computing gradient and 
	projections. On the other hand, the third task (which depends on $\omega$) is more involved and 
	might requires a nested optimization algorithm . Thus, in many scenarios, we end up with 
	\emph{nested} schemes which implies that: (i) there is accumulation of computational error in each 
	step, and (ii) the stopping criteria of the nested algorithm at each iteration is not well-
	defined. Therefore, the third task determines the computational complexity of the entire method, 
	and thus the applicability of the MNG method for certain implementation.
\medskip

	In the case where $\omega\left(\cdot\right) := \norm{\cdot}_\bbq^{2}$, where $\bbq$ is a positive 
	definite matrix, the computation is easy and given by an explicit formula as noted in 
	\cite[Example 1, Page 36]{BS2014}, although it may require some decomposition and inversion of 
	matrix $\bbq$ (see Section \ref{Sec:Numerical Experiments} for more details).
\medskip

	The main result derived in \cite{BS2014} (\cf \cite[Theorems 4.1 and 4.2, Pages 37 and 39]{BS2014}) is valid when Assumptions \ref{A:Composite} and \ref{A:Omega} hold and when $g\left(\cdot
	\right) := \delta_{X}\left(\cdot\right)$. As we already mentioned, in \cite{BS2014} the authors 
	did not assume that $\nabla \omega$ is Lipschitz continuous, only continuously differentiable. We 
	state here the following result which deals with the case where the Lipschitz constant $L_{f}$ of 
	$\nabla f$ is known (for a backtracking version see \cite{BS2014}).
	\begin{proposition} \label{P:MNG} 
		Let $\Seq{\bx}{k}$ be the sequence generated by the MNG method. Then, the sequence $\Seq{\bx}	
		{k}$ converges to the optimal solution $\bx_{mn}^{\ast}$ of problem \eqref{Prob:MNP} and, for 
		any $k \in \nn$, we have that
		\begin{equation*}
			\min_{1 \leq l \leq k} \varphi\left(T_{1/L_{f}}\left(\bx^{l}\right)\right) - \varphi
			\left(\bx_{mn}^{\ast}\right) \leq \frac{\rho L_{f}\norm{\bx^{0} - \bx_{mn}^{\ast}}^{2}}
			{\sqrt{k}},
		\end{equation*}
		where $\rho = 1$ if $X = \real^{n}$ and $\rho = 4/3$ otherwise.
	\end{proposition}
	It should be noted that the MNG method is not a feasible method in the sense that $\bx^{k}$, $k 
	\in \nn$, does not necessarily belongs to the constraint set $X$ and therefore the rate of 
	convergence result is obtained on the feasible sequence $\left\{ T_{1/L_{f}}\left(\bx^{k}\right) 
	\right\}_{k \in \nn}$, $k \in \nn$. Furthermore, though in the original paper the authors only 
	discuss the case where $g\left(\cdot\right) := \delta_{X}\left(\cdot\right)$ the result can 
	actually be extended to the more general case given in the introduction.
\medskip

	Our main goal in this paper is to study a different algorithm for solving bi-level optimization 
	problems, than the MNG method, for which \textit{we prove a rate of convergence} in terms of the 
	inner objective function values that is superior to the rate of the MNG method (given in 
	Proposition \ref{P:MNG} above). In addition to complexity aspect, the studied method, which is 
	discussed in the next section, \textit{is simpler and capable of tackling bi-level problems for 
	which the outer	objective function is not necessarily smooth}. This attractiveness is mainly due 
	to the fact that the studied method does not require the minimization of the outer objective 
	function $\omega$ over two half-spaces as needed in the MNG method. 
 
\section{The Sequential Averaging Method} \label{Sec:SAM}

\subsection{The General Framework} \label{SSec:XuMethod}
	Our approach here is based on taking an existing algorithm, that we call Sequential Averaging 
	Method (SAM), which was developed in \cite{Xu2004} for solving a certain class of fixed-point 
	problems (see precise details below), and determining how it can be used in the setting of bi-
	level optimization problems as described in Section \ref{Sec:AlgFrameMathTools}. It should be 
	noted that in \cite{Xu2004} it is already proved that SAM generates a sequence which converges to 
	a solution of the corresponding fixed-point problem. Now we will discuss in detail the method 
	proposed in \cite{Xu2004}. 
\medskip

	The problem of main interest in \cite{Xu2004} is finding a fixed-point of the nonexpansive mapping
	$T$, that is, $\bx^{\ast} \in \Fix(T)$, which also satisfies certain property with respect to a
	contraction mapping $S$ over all points which belong to $\Fix(T)$. This property is formulated 
	using the following variational equality
	\begin{equation} \label{VariationalInequality}
		\act{\bx^{\ast} - S\left(\bx^{\ast}\right) , \bx - \bx^{\ast}} \geq 0, \quad \forall \,\, \bx
		\in \Fix(T).
	\end{equation}
	This means that the problem here is to find a fixed-point of the mapping $T$ which is ``better" 
	than all other fixed-points of $T$ in the sense of inequality \eqref{VariationalInequality}.
\medskip
	
	The SAM iteratively generates a sequence $\Seq{\bx}{k}$ starting from any $\bx^{0} \in \real^{n}$ 
	by averaging the two mappings $S$ and $T$ in the following way
	\begin{equation*}
		\bx^{k} = \alpha_{k}S\left(\bx^{k - 1}\right) + \left(1 - \alpha_{k}\right)T\left(\bx^{k - 1}
		\right),
	\end{equation*}
	where $\seq{\alpha}{k}$ is a well-chosen sequence of real numbers from $\left(0 , 1\right]$ which 
	satisfies the following assumption.
	\begin{assumption} \label{A:AlphSeq}
		Let $\seq{\alpha}{k}$ be a sequence of real numbers in $\left(0 , 1\right]$ which satisfies 
		$\lim_{k \rightarrow \infty} \alpha_{k} = 0$, $\sum_{k = 1}^{\infty} \alpha_{k} = \infty$ and 
		$\lim_{k \rightarrow \infty} \alpha_{k + 1}/\alpha_{k} = 1$.		 
	\end{assumption}
	It should be noted that Assumption \ref{A:AlphSeq} holds true for several choices of sequences 
	$\seq{\alpha}{k}$ which include, for example, $\alpha_{k} = \alpha/k$, $k \in \nn$ for any choice 
	of $\alpha \in \left(0 , 1\right]$. 
\medskip

	The following lemma summarizes the main known results on SAM, which were proved in
	\cite[Theorem 3.2]{Xu2004} and formed the basis for this paper.
	\begin{lemma} \label{L:Xu}	
		Assume that $S : \real^{n} \rightarrow \real^{n}$ is a $\beta$-contraction and that $T :
		\real^{n} \rightarrow \real^{n}$ is nonexpansive mapping, for which $\Fix(T) \neq \emptyset$. 
		Let $\Seq{\bx}{k}$ be the sequence generated by SAM. If Assumption \ref{A:AlphSeq} holds true, 
		then the following assertions are valid.
		\begin{itemize}
			\item[$\rm{(i)}$] The sequence $\Seq{\bx}{k}$ is bounded, in particular, for any 
				$\tilde{\bx} \in \Fix(T)$ we have, for all $k \in \nn$, that
				\begin{equation} \label{D:Cx}
					\norm{\bx^{k} - \tilde{\bx}} \leq C_{\tilde{\bx}} := \max \left\{ \norm{\bx^{0} -
					\tilde{\bx}} , \frac{1}{1 - \beta}\norm{\left(I - S\right)\tilde{\bx}}\right\}.
				\end{equation}	
				Moreover, for all $k \in \nn$, we also have that 
				\begin{equation*}
					\norm{T\left(\bx^{k}\right) - \tilde{\bx}} \leq C_{\tilde{\bx}} \quad \text{and} 
					\quad \norm{S\left(\bx^{k}\right) - S\left(\tilde{\bx}\right)} \leq \beta 
					C_{\tilde{\bx}}.
				\end{equation*}
			\item[$\rm{(ii)}$] The sequence $\Seq{\bx}{k}$ converges to some $\bx^{\ast} \in \Fix(T)$.
			\item[$\rm{(iii)}$] The limit point $\bx^{\ast}$ of $\Seq{\bx}{k}$, which the existence is 
				ensured by (ii), satisfies the following variational inequality
				\begin{equation} \label{L:Xu:VariationalInequality}
					\act{\bx^* - S\left(\bx^{\ast}\right) , \bx - \bx^{\ast}} \geq 0, \quad \forall \,
					\, \bx \in \Fix(T).
				\end{equation}
		\end{itemize} 
	\end{lemma}
	We conclude this part by highlighting and streamlining our contributions in this paper, which go 
	beyond convergence of SAM as recorded in Lemma \ref{L:Xu}.
	\begin{itemize}
		\item[$\rm{(i)}$] We prove that under a specific choice of parameters, SAM generates a 
			sequence $\Seq{\bx}{k}$ for which the sequence of the gaps between the iterator and its 
			mapping by $T$, that is $\left\{ \norm{T\left(\bx^{k}\right) - \bx^{k}} \right\}_{k \in 
			\nn}$, converges with the non-asymptotic rate of $O(1/k)$. \textit{This result gives a 
			rate of convergence to a fixed point of $T$ for the first time}.
		\item[$\rm{(ii)}$] We study BiG-SAM for solving bi-level optimization problems, for which the 
			functions $f$, $g$ and $\omega$ satisfy Assumptions \ref{A:Composite} and \ref{A:Omega}. 
			For this variant, we prove an $O(1/k)$ rate of convergence for the sequence of inner 
			objective function values (see details in Section \ref{Sec:ConvergenceAnalysis}). 
			\textit{This result affirmatively answers the question raised in \cite{BS2014} about a 
			first order method for bi-level problems with an improved rate of convergence.}
		\item[$\rm{(iii)}$] We show that BiG-SAM can be also applied in situations where the outer 
			objective function $\omega$ satisfies only Assumption \ref{A:Omega}(i) but not 
			\ref{A:Omega}(ii), \ie it is strongly convex but not necessarily smooth. In this case we 
			also prove a rate of convergence result in terms of the inner objective function values 
			(see details in Section \ref{SSec:SAMNonsmooth}).
	\end{itemize}
	
\subsection{SAM for Smooth Bi-level Optimization Problem} \label{SSec:SAMSmooth}
	We begin this part by connecting the fixed-point problem discussed above with the bi-level 
	optimization problem described in Section \ref{SSec:BiLevelOpt}. We will make this connection by 
	linking the mappings $S$ and $T$ with problems \eqref{Prob:MNP} and \eqref{Prob:P}, respectively. 
	We begin by connecting the mapping $T$ with problem \eqref{Prob:P}.
\medskip
	
	First of all, as explained in Section \ref{SSec:XuMethod}, the mapping $T$ and its fixed-point set 
	$\Fix(T)$ are the inner part in the fixed-point problem, in the sense that we want to find a 
	fixed-point of $T$ which is ``better" than any other points in $\Fix(T)$ with respect to a 
	criteria given by the mapping $S$ (see \eqref{VariationalInequality}). In the bi-level setting, 
	the situation is similar, since the inner problem \eqref{Prob:P} is actually the inner part and 
	from all its optimal solutions, \ie from the set $X^{\ast}$, we would like to find the one which 
	satisfies the additional criteria, being minimizer of $\omega$ over $X^{\ast}$. Therefore, the 
	following relations 	hold.
	\begin{itemize}
		\item[$\rm{(i)}$] The mapping $T$ and its fixed-point set $\Fix(T)$ are related to problem
			\eqref{Prob:P} with the composite function $\varphi = f + g$ and the optimal solution set 
			$X^{\ast}$.
		\item[$\rm{(ii)}$] The mapping $S$ is related to problem \eqref{Prob:MNP} and the objective 
			function $\omega$.
	\end{itemize}
	From now on, we set the mapping $T$ to be the prox-grad mapping defined in \eqref{ProxGradMap}, 
	that is, for some $t \in \left(0 , {1}/{L_{f}}\right]$ we have 
	\begin{equation} \label{TChoice}
		T\left(\bx\right) := T_{t}\left(\bx\right)  =\prox_{tg}\left(\bx - t\nabla f\left(\bx\right)
		\right).
	\end{equation}
	According to Lemma \ref{L:ProxGradProp} and since Assumption \ref{A:Composite} holds, we ensure 
	that in this case $T$ is nonexpansive and $\Fix(T) = X^{\ast}$. We therefore fill all the 
	requirements on the mapping $T$ in Lemma \ref{L:Xu}, and immediately obtain from Lemma 
	\ref{L:Xu}(ii) that the sequence generated by SAM (with any $\beta$-contraction $S$) converges to 
	a point in $X^{\ast}$. Thus, the only remaining part is to connect problem \eqref{Prob:MNP} with 
	the criteria given in the variational inequality presented in Lemma \ref{L:Xu}(iii).
\medskip

	Taking into account Proposition \ref{P:GradContraction} and given that Assumption \ref{A:Omega} 
	holds, a natural choice for the mapping $S$ is as follows
	\begin{equation} \label{SChoice}
		S\left(\bx\right) := \bx - s\nabla \omega\left(\bx\right),
	\end{equation}
	where $s \in \left(0 , 2/\left(\sigma + L_{\omega}\right)\right]$. In this case we know, from 
	Proposition \ref{P:GradContraction}, that $S$ is a $\beta$-contraction with $\beta = \left({1 - 
	{2sL_{\omega}\sigma}/{\left(L_{\omega} + \sigma\right)}}\right)^{1/2}$.
\medskip

	Therefore, SAM for solving the bi-level optimization problems \eqref{Prob:P} and \eqref{Prob:MNP} 
	is given now.
\vspace{0.2in}

	{\center\fbox{\parbox{16cm}{{\bf \underline{Bi-level Gradient SAM (BiG-SAM)}}
		\begin{itemize}
       		\item[$\rm{(1)}$] {\bf Input:} $t \in \left(0 , 1/L_{f}\right]$, $s \in \left(0 , 2/
       			\left(L_{\omega} + \sigma\right)\right]$, and $\seq{\alpha}{k}$ satisfying Assumption 
       			\ref{A:AlphSeq}.
            \item[$\rm{(2)}$] {\bf Initialization}: $\bx^{0} \in \real^{n}$.
            \item[$\rm{(3)}$] {\bf General Step} ($k = 1 , 2 , \ldots$):
            		\begin{align}
			 		\by^{k} &= \prox_{tg}\left(\bx^{k - 1} - t\nabla f\left(\bx^{k - 1}\right)
			 		\right),\\
			 		\bz^{k} &= \bx^{k - 1} - s\nabla \omega\left(\bx^{k - 1}\right),
			 		\label{OmegaGradStep} \\		
					\bx^{k} &= \alpha_{k}\bz^{k} + \left(1 - \alpha_{k}\right)\by^{k}.
				\end{align}
        \end{itemize}\vspace{-0.2in}}}}

\vspace{0.2in}

	To conclude this section we would like to interpret the variational inequality given in Lemma 
	\ref{L:Xu}(iii) in the setting of bi-level optimization, in which $S$ and $T$ are given by 
	\eqref{SChoice} and \eqref{TChoice}, respectively, for some $t \in \left(0 , 1/L_{f}\right]$ and 
	$s \in \left(0 , 2/\left(L_{\omega} + \sigma\right)\right]$. In the following result we give the 
	desired interpretation and prove that BiG-SAM generates a sequence which converges to the solution 
	of problem \eqref{Prob:MNP}.
	\begin{proposition} \label{P:ConvergenceBiGSAM}
		Let $\Seq{\bx}{k}$ be a sequence generated by BiG-SAM and suppose that Assumptions
		\ref{A:Composite}, \ref{A:Omega} and \ref{A:AlphSeq} hold true. Then, the sequence 
		$\Seq{\bx}{k}$ converges to $\bx^{\ast} \in X^{\ast}$ which satisfies
		\begin{equation} \label{VarInequalityBiLevel}
			\act{\nabla \omega\left(\bx^{\ast}\right) , \bx - \bx^{\ast}} \geq 0, \quad \forall \,\, 
			\bx \in X^{\ast},
		\end{equation}
		and therefore, $\bx^{\ast} = \bx_{mn}^{\ast}$ is the optimal solution of problem 
		\eqref{Prob:MNP}.
	\end{proposition}
	\begin{proof}		
		Since Assumptions \ref{A:Composite} and \ref{A:Omega} hold true, by Lemma \ref{L:ProxGradProp} 
		and Proposition \ref{P:GradContraction}, we have that $S$ and $T$ which defined in 
		\eqref{SChoice} and \eqref{TChoice} are a contraction and a nonexpansive mapping, 
		respectively. Thus, all the assumptions of Lemma \ref{L:Xu} are valid and therefore we 
		immediately obtain that $\Seq{\bx}{k}$ converges to $\bx^{\ast} \in X^{\ast}$ (see Lemmas 
		\ref{L:ProxGradProp}(ii) and \ref{L:Xu}(ii)). The only remaining part is showing that the 
		variational inequality given in Lemma \ref{L:Xu}(iii) implies that 
		\eqref{VarInequalityBiLevel} holds true. Indeed, using the fact that $S = I - s\nabla \omega$ 
		we obtain that \eqref{L:Xu:VariationalInequality} is equivalent to
		\begin{equation*}
			\act{\bx^{\ast} - \left(\bx^{\ast} - s\nabla \omega\left(\bx^{\ast}\right)\right) , \bx -
			\bx^{\ast}} \geq 0, \quad \forall \,\, \bx\in X^{\ast},
		\end{equation*}
		which directly implies that \eqref{VarInequalityBiLevel} holds true, since $s > 0$. This means 
		that $\bx^{\ast}$ satisfies the first order optimality condition for constrained convex 
		problems (see, for example, \cite[Proposition 2.1.2, Page 194]{B99}) and therefore $\bx^{\ast} 
		= \bx_{mn}^{\ast}$, as asserted.
	\end{proof}
	
\section{Rate of Convergence Analysis} \label{Sec:ConvergenceAnalysis}
	In this section we will first prove a technical result about the rate of convergence of the gap 
	between two successive iterations generated by SAM in its most generality, for solving the 
	fixed-point problem, as described in Section \ref{SSec:XuMethod}. Then, we will use it to derive 
	the main result of our paper which is a rate of convergence result for BiG-SAM in terms of the 
	inner objective function values. This rate is superior to the one presented in \cite{BS2014} for 
	the case of differentiable $\omega$, and holds true for any contraction mapping $S$ regardless of 
	$\omega$.
\medskip

	We first present a technical lemma which will assist us in the rate of convergence proof. The 
	proof of this lemma is given in Appendix \ref{A:ProofLemmaAlpha}.
	\begin{lemma} \label{L:Alpha}
		Let $M > 0$. Assume that $\seq{a}{k}$ is a sequence of nonnegative real numbers which satisfy 
		$a_{1} \leq M$ and 	
		\begin{equation*}
			a_{k + 1} \leq \left(1 - \gamma b_{k + 1}\right)a_{k} + \left(b_{k} - b_{k + 1}\right) 
			c_{k}, \quad k \geq 1,
		\end{equation*}
		where $\gamma \in \left(0 , 1\right]$, $\seq{b}{k}$ is a sequence which is defined by $b_{k} = 
		\min\left\{ {2}/\left(\gamma k\right) , 1 \right\}$, and $\seq{c}{k}$ is a sequence of real 
		numbers such that $c_{k} \leq M < \infty$. Then, the sequence $\seq{a}{k}$ satisfies
		\begin{equation*}
			a_{k} \leq \frac{MJ}{\gamma k}, \quad k \geq 1,
		\end{equation*}
		where $J = \lfloor 2/\gamma \rfloor$.
	\end{lemma}
	For ease of notation, from this point onward, we will denote, for any $k \in \nn$, $\by^{k} = T
	\left(\bx^{k - 1}\right)$ and $\bz^{k} = S\left(\bx^{k - 1}\right)$. For convenience we will split 
	the rate analysis into two technical results which will lead us to the main result given in 
	Theorem \ref{T:FunctionRate}. In Lemma \ref{L:Usefulinequalities} we present some useful 
	inequalities, and in Lemma \ref{L:ConvergenceRate} we show that by choosing an appropriate 
	sequence $\seq{\alpha}{k}$ we can bound the distance between two successive elements of the 
	sequence $\Seq{\bx}{k}$ by $O(1/k)$ and to show that the sequence $\Seq{\bx}{k}$ converges, with 
	the same rate, to a fixed-point of $T$.
	\begin{lemma} \label{L:Usefulinequalities}
		Assume that $S : \real^{n} \rightarrow \real^{n}$ is a $\beta$-contraction and that $T :
		\real^{n} \rightarrow \real^{n}$ is nonexpansive mapping, for which $\Fix(T) \neq \emptyset$. 
		Let $\Seq{\bx}{k}$, $\Seq{\by}{k}$ and $\Seq{\bz}{k}$ be sequences generated by SAM. Then, for 
		any $k \geq 1 $ and any $\tilde{\bx} \in \Fix(T)$, defining $\tilde{\bz} = S\left(\tilde{\bx}
		\right)$ the following inequalities hold true
		\begin{align}
			\norm{\by^{k + 1} - \by^{k}} & \leq \norm{\bx^{k} - \bx^{k - 1}}, \label{NonexpansiveYk} 
			\\
			\norm{\bz^{k + 1} - \bz^{k}} & \leq \beta\norm{\bx^{k} - \bx^{k - 1}}, 
			\label{NonexpansiveZk} \\
			\norm{\by^{k} - \tilde{\bx}} & \leq \norm{\bx^{k - 1} - \tilde{\bx}},
			\label{NonexpansiveYstar} \\
			\norm{\bz^{k} - \tilde{\bz}} & \leq \beta\norm{\bx^{k - 1} - \tilde{\bx}}.					
			\label{NonexpansiveZstar}
		\end{align}
	\end{lemma}
	\begin{proof}		
		All the required inequalities are a direct consequence of the nonexpansivity of $T$ and the 
		contraction property of $S$.
	\end{proof}
	Now we prove the rate of convergence to zero of the sequence $\left\{ \norm{\bx^{k} - \bx^{k - 1}} 
	\right\}_{k \in \nn}$, where $\Seq{\bx}{k}$ is generated by SAM and the averaging parameters 
	$\alpha_{k}$, $k \in \nn$, are chosen as follows
	\begin{equation} \label{ChooseAlpha}
		\alpha_{k} = \min \left\{ \frac{2\gamma}{k\left(1 - \beta\right)} , 1 \right\}, \quad k \geq 
		1,
	\end{equation}
	where $\gamma \in \left(0 , 1\right]$. For the simplicity of the developments, we will prove our 
	results when $\gamma = 1$. It should be noted that all the results below remains valid also when $
	\gamma$ is chosen arbitrarily from the interval $\left(0 , 1\right]$.
	\begin{lemma} \label{L:ConvergenceRate}
		Let $\Seq{\bx}{k}$, $\Seq{\by}{k}$ and $\Seq{\bz}{k}$ be sequences generated by SAM where 
		$\seq{\alpha}{k}$ is defined by \eqref{ChooseAlpha}. Then, for any $\tilde{\bx} \in \Fix(T)$, 
		the two sequences $\left\{ \norm{\bx^{k} - \bx^{k - 1}} \right\}_{k \in \nn}$ and $\left\{ 
		\norm{\by^{k} - \bx^{k - 1}} \right\}_{k \in \nn}$ converge to $\bo$, and  the rates of 
		convergence are given by
		\begin{equation} \label{T:ConvergenceRate:0}
			\norm{\bx^{k} - \bx^{k - 1}} \leq \frac{2C_{\tilde{\bx}}J}{\left(1 - \beta\right)k}, \quad 
			k \geq 1,
		\end{equation}
		and
		\begin{equation} \label{T:ConvergenceRate:00}
			\norm{\by^{k} - \bx^{k - 1}} \leq \frac{2C_{\tilde{\bx}}\left(J + 2\right)}{\left(1 -
			\beta\right)k}, \quad k \geq 1,
		\end{equation}		 
		where $C_{\tilde{\bx}}$ is defined in \eqref{D:Cx}, and $J = \lfloor 2/\left(1 - \beta\right) 
		\rfloor$.				
	\end{lemma}
	\begin{proof}
		By the definition of $\bx^{k}$ and $\bx^{k + 1}$ we immediately obtain
		\begin{align}
			\norm{\bx^{k + 1} - \bx^{k}} & = \norm{\left(1 - \alpha_{k + 1}\right)\by^{k + 1} + 
			\alpha_{k + 1}\bz^{k + 1} - \left(\left(1 - \alpha_{k}\right)\by^{k} + \alpha_{k}\bz^{k} 
			\right)} \nonumber \\
			& = \norm{\left(1 - \alpha_{k + 1}\right)\left(\by^{k + 1} - \by^k\right) + \alpha_{k + 1}
			\left(\bz^{k + 1} - \bz^{k}\right) + \left(\alpha_{k} - \alpha_{k + 1}\right)\left(\by^{k} 
			- \bz^{k}\right)} \nonumber \\			
			& \leq \left(1 - \alpha_{k + 1}\right)\norm{\by^{k + 1} - \by^{k}} + \alpha_{k + 1} 
			\norm{\bz^{k + 1} - \bz^{k}} + \left(\alpha_{k} - \alpha_{k + 1}\right)\norm{\by^{k} -
			\bz^{k}} \nonumber \\
			&\leq \left(1 - \alpha_{k + 1}\right)\norm{\bx^{k} - \bx^{k - 1}} + \alpha_{k + 1}\beta 
			\norm{\bx^{k} - \bx^{k - 1}} + \left(\alpha_{k} - \alpha_{k + 1}\right)\norm{\by^{k} -
			\bz^{k}} \nonumber \\
			& = \left(1 - \alpha_{k + 1}\left(1 - \beta\right)\right)\norm{\bx^{k} - \bx^{k - 1}} +
			\left(\alpha_{k} - \alpha_{k + 1}\right){\norm{\by^{k} - \bz^{k}}}, 
			\label{T:ConvergenceRate:1}
		\end{align}
		where the second inequality follows from \eqref{NonexpansiveYk} and \eqref{NonexpansiveZk}. 
		Now, let $\tilde{\bx} \in \Fix(T)$ and let $\tilde{\bz} = S\left(\tilde{\bx}\right)$, then
		\begin{align}
			\norm{\by^{k} - \bz^{k}} & = \norm{\by^{k} - \tilde{\bx} + \tilde{\bx} - \tilde{\bz} + 
			\tilde{\bz} - \bz^{k}} \nonumber \\
			& \leq \norm{\by^{k} - \tilde{\bx}} + \norm{\tilde{\bx} - \tilde{\bz}} + \norm{\tilde{\bz} 
			- \bz^{k}} \nonumber \\
			& \leq \norm{\bx^{k - 1} - \tilde{\bx}} + \norm{\left(I - S\right)\tilde{\bx}} + \beta
			\norm{\bx^{k - 1} - \tilde{\bx}} \nonumber \\
			& \leq C_{\tilde{\bx}} + \left(1 - \beta\right)C_{\tilde{\bx}} + \beta C_{\tilde{\bx}} = 
			2C_{\tilde{\bx}}, \label{T:ConvergenceRate:2}
		\end{align}
		where the second inequality follows from \eqref{NonexpansiveYstar} and 
		\eqref{NonexpansiveZstar}, as well as the definition of $\tilde{\bz}$, and the last inequality 
		follows from Lemma \ref{L:Xu}(i). Additionally, we have that
		\begin{equation}\label{T:ConvergenceRate:3}
			\norm{\bx^{1} - \bx^{0}} = \norm{\bx^{1} - \tilde{\bx} + \tilde{\bx} - \bx^{0}} \leq 
			\norm{\bx^{1} - \tilde{\bx}} + \norm{\bx^{0} - \tilde{\bx}} \leq 2C_{\tilde{\bx}},
		\end{equation}
		where the second inequality follows from Lemma \ref{L:Xu}(i). The convergence rate for the 
		sequence $\left\{ \norm{\bx^{k} - \bx^{k - 1}} \right\}_{k \in \nn}$ is now an immediate 
		result of applying Lemma \ref{L:Alpha} on \eqref{T:ConvergenceRate:1} with $a_{k} := 
		\norm{\bx^{k} - \bx^{k - 1}}$, $b_{k} := \alpha_{k}$, $\gamma := 1 - \beta$ and $c_{k} := 
		\norm{\by^{k} - \bz^{k}}$ and using \eqref{T:ConvergenceRate:2} and 
		\eqref{T:ConvergenceRate:3} where $M := 2C_{\tilde{\bx}}$. This proves 
		\eqref{T:ConvergenceRate:0}.
\medskip
		
		The convergence rate for $\left\{ \norm{\by^{k} - \bx^{k - 1}} \right\}_{k \in \nn}$ is 
		derived by the following arguments
		\begin{align*}
			\norm{\by^{k} - \bx^{k - 1}} & = \norm{\by^{k} - \bx^{k} + \bx^{k} - \bx^{k - 1}} \\
			& \leq \norm{\by^{k} - \bx^{k}} + \norm{\bx^{k} - \bx^{k - 1}} \\
			& = \alpha_{k}\norm{\by^{k} - \bz^{k}} + \norm{\bx^{k} - \bx^{k - 1}} \\
			& \leq\frac{2}{\left(1 - \beta\right)k}2C_{\tilde{\bx}} + \frac{2C_{\tilde{\bx}}J}{\left(1 
			- \beta\right)k} \\
			& = \frac{2C_{\tilde{\bx}}\left(J + 2\right)}{\left(1 - \beta\right)k},
		\end{align*}		
		where the second inequality is due to the previous result as well as 
		\eqref{T:ConvergenceRate:2}, which was already proven.
	\end{proof}			
	It should be noted that in the case of BiG-SAM, the rate of convergence proven in 
	\eqref{T:ConvergenceRate:00} means that the sequence $\Seq{\bx}{k}$ converges to an optimal 
	solution of the inner problem (P) with a rate of $O(1/k)$.
\medskip

	Now we turn to discuss the main result which is convergence of the sequence $\left\{ \varphi
	\left(\by^{k}\right) \right\}_{k \in \nn}$. We discuss the convergence of this sequence rather 
	than the convergence of the sequence $\left\{ \varphi\left(\bx^{k}\right) \right\}_{k \in \nn}$ 
	since the latter might be an infeasible in terms of the domain of the function $\varphi$ (see also 
	\cite{BS2014}). Moreover, since we proved that $\norm{\by^{k} - \bx^{k - 1}}\rightarrow 0$ as $k 
	\rightarrow \infty$ and $\varphi$ is lower semicontinuous it follows that proving convergence of 
	the sequence $\left\{ \varphi\left(\by^{k}\right) \right\}_{k \in \nn}$ to the optimal value also 
	implies convergence of the sequence $\left\{ \varphi\left(\bx^{k}\right) \right\}_{k \in \nn}$ to 
	the same value. 
	\begin{theorem} \label{T:FunctionRate}
		Let $\Seq{\bx}{k}$, $\Seq{\by}{k}$ and $\Seq{\bz}{k}$ be sequences generated by BiG-SAM where 
		$\seq{\alpha}{k}$ is defined by \eqref{ChooseAlpha}. Then, for all $t \leq 1/L_{f}$ and $k \in 
		\nn$, we have that
		\begin{equation*}
			\varphi\left(\by^{k}\right) - \varphi\left(\bx_{mn}^{\ast}\right) \leq \frac{2C_{\bx_{mn}
			^{\ast}}^{2}\left(J + 2\right)}{\left(k + 1\right)\left(1 - \beta\right)t},
		\end{equation*}
		where $C_{\bx_{mn}^{\ast}}$ is defined in \eqref{D:Cx} and $J = \lfloor 2/\left(1 - \beta
		\right) \rfloor$.					
	\end{theorem}
	\begin{proof}
		From Proposition \ref{P:ProximalInequality} we have, for any step-size $t \leq 1/L_{f}$, that 
		the following inequality holds true
		\begin{equation} \label{C:FunctionRate:1}
			\varphi\left(\by^{k + 1}\right) - \varphi\left(\bx_{mn}^{\ast}\right) \leq \frac{1}{t}
			\act{\bx^{k} - \by^{k + 1} , \bx^{k} - \bx_{mn}^{\ast}} - \frac{1}{2t}\norm{\bx^{k} - 
			\by^{k + 1}}^{2}.
		\end{equation}
		Applying Lemmas \ref{L:Xu}(i) and \ref{L:ConvergenceRate} for $\bx_{mn}^{\ast} \in X^{\ast} = 
		\Fix(T_{t})$ we obtain that
		\begin{equation} \label{C:FunctionRate:2}
			\act{\bx^{k} - \by^{k + 1} , \bx^{k} - \bx_{mn}^{\ast}} \leq \norm{\bx^{k} - \by^{k + 1}}
			\cdot \norm{\bx^{k} - \bx_{mn}^{\ast}} \leq \frac{2C_{\bx_{mn}^{\ast}}^{2}\left(J + 
			2\right)}{\left(1 - \beta\right)\left(k + 1\right)}.
		\end{equation}	
		Substituting \eqref{C:FunctionRate:1} back into \eqref{C:FunctionRate:2} we obtain that
		\begin{equation*}
			\varphi\left(\by^{k + 1}\right) - \varphi\left(\bx_{mn}^{\ast}\right) \leq 
			\frac{2C_{\bx_{mn}^{\ast}}^{2}\left(J + 2\right)}{\left(k + 1\right)\left(1 - \beta
			\right)t},
		\end{equation*}
		which proves the desired result.				
	\end{proof}
	\begin{remark}
		The step-size $s$, which is used in step \eqref{OmegaGradStep}, should be chosen such that the 
		mapping $S$ is a contraction. According to Proposition \ref{P:GradContraction} the step-size 
		$s$ depends on the knowledge of $L_{\omega}$ and $\sigma$, or at least on an upper bound on
		$L_{\omega} + \sigma$. Moreover, in order to calculate $\alpha_{k}$, $k \in \nn$, we need an 
		upper bound on the contraction parameter $\beta$ which also depends on $L_{\omega}$ and 
		$\sigma$ or a lower bound on $\sigma L_{\omega}$.
	\end{remark}

\subsection{SAM for Nonsmooth Bi-level Optimization Problems} \label{SSec:SAMNonsmooth}
	In this section we focus on problem \eqref{Prob:MNP} as described in Section \ref{SSec:BiLevelOpt} 
	for which the objective function $\omega$ does not necessarily satisfy Assumption \ref{A:Omega}. 
	Here we replace it by the following milder assumption. 
	\setcounter{assumption}{1}
	\renewcommand{\theassumption}{\Alph{assumption}'}
	\begin{assumption} \label{A:OmegaNonsmooth}
		 $\omega : \real^{n} \rightarrow \real$ is strongly convex with parameter $\sigma > 0$ and 
		 $\ell_{\omega}$-Lipschitz continuous.
	\end{assumption}
	It is clear that BiG-SAM can not be applied to bi-level problems for which $\omega$ satisfies 
	Assumption \ref{A:OmegaNonsmooth} instead of Assumption \ref{A:Omega}, since $\omega$ is not 
	necessarily differentiable. However, due to the strong convexity of $\omega$ we may use BiG-SAM in 
	the following way.
\medskip
	
	We will use the Moreau envelope $M_{s\omega}$ of $\omega$ as a smooth replacement of the original 
	objective function $\omega$. As we have already mentioned in Section \ref{SSec:BiLevelOpt}, the 
	Moreau envelope is continuously differentiable, its gradient is Lipschitz continuous with constant 
	$1/s$ and strongly convex (see Proposition \ref{P:StrongConvexMoreauEnvelope}). Based on these 
	facts we obtain that $M_{s\omega}$ satisfies Assumption \ref{A:Omega} and therefore BiG-SAM can be 
	applied in this case on the Moreau envelope $M_{s\omega}$. It should be noted that in this case 
	step \eqref{OmegaGradStep} is given by
	\begin{equation} \label{ComputeZnonsmooth}
		\bz^{k} = \bx^{k - 1} - s\nabla M_{s\omega}\left(\bx^{k - 1}\right) = \bx^{k - 1} - s\frac{1}
		{s}\left(\bx^{k - 1} - \prox_{s\omega}\left(\bx^{k - 1}\right)\right) = \prox_{s\omega}
		\left(\bx^{k - 1}\right),
	\end{equation}
	where the second equality follows from \eqref{MoreauEnvelopeGradient}. This means that in order to 
	obtain $\bz^{k}$, $k \in \nn$, we need to compute the proximal mapping of $\omega$.
	\begin{remark}
		Based on the equality given in \eqref{ComputeZnonsmooth} it can be seen that BiG-SAM applied 
		on the Moreau envelope $M_{s\omega}$ is exactly SAM which is applied to the bi-level problem 
		with $S$ being the proximal mapping of $\omega$. In this respect it should be noted that the 
		proximal mapping of a strongly convex function is a contraction (see Lemma 
		\ref{L:StronglyConvexProxContraction} in Appendix \ref{A:ProofStrongMoreau}) and therefore all 
		the theory presented in Section \ref{SSec:XuMethod} is valid in this setting too.
	\end{remark}
	The following result is an immediate consequence of Proposition \ref{P:ConvergenceBiGSAM} applies 
	on the following mappings
	\begin{equation*}
		S\left(\bx\right) = \bx - s\nabla M_{s\omega}\left(\bx\right) \quad \text{and} \quad T
		\left(\bx\right) = \prox_{tg}\left(\bx - t\nabla f\left(\bx\right)\right),
	\end{equation*}
	where $s > 0$ and $t \in \left(0 , 1/L_{f}\right]$. 
	\begin{proposition} \label{P:ConvergenceMoreauOptimal}
		Let $\Seq{\bx}{k}$ be a sequence generated by BiG-SAM and suppose that Assumptions
		\ref{A:Composite}, \ref{A:OmegaNonsmooth} and \ref{A:AlphSeq} hold true and let $s > 0$. Then,
		the sequence $\Seq{\bx}{k}$ converges to $\bx_{s}^{\ast} \in X^{\ast}$ which satisfies
		\begin{equation} \label{MoreauOptimality}
			\act{\nabla M_{s\omega}\left(\bx_{s}^{\ast}\right) , \bx - \bx_{s}^{\ast}} \geq 0, \quad 
			\forall \,\, \bx \in X^{\ast}.
		\end{equation}
		Therefore, $\bx_{s}^{\ast}$ is the optimal solution of problem \eqref{Prob:MNP} with respect 
		to the Moreau envelope $M_{s\omega}$, \ie
		\begin{equation*}
			\bx_{s}^{\ast} = \argmin_{\bx \in X^{\ast}} M_{s\omega}\left(\bx\right),
		\end{equation*}
		where $X^{\ast}$ is the optimal solutions set of problem \eqref{Prob:P}.
	\end{proposition}
	\begin{proof}
		Similar to the proof of Proposition \ref{P:ConvergenceBiGSAM} using \eqref{ComputeZnonsmooth}.
	\end{proof}
	The section began with the goal of solving problem \eqref{Prob:MNP} for which $\omega$ is not 
	necessarily smooth. To this end we suggested to apply BiG-SAM on the Moreau envelope $M_{s\omega}$ 
	for some step-size $s > 0$ and as a result we get $\bx_{s}^{\ast}$ which minimizes $M_{s\omega}$ 
	over $X^{\ast}$. The step-size $s$ also plays an important role in controlling the distance 
	between $\bx_{s}^{\ast}$ and the original solution $\bx_{mn}^{\ast}$, this will be made precise 
	below.
\medskip

	The fact that we smoothed the outer objective function $\omega$ seems to not influence the rate of 
	convergence result which is in terms of the inner objective function. However, a careful 
	inspection shows that this is not really the case since the rate of convergence result (see 
	Theorem \ref{T:FunctionRate}) depends on the contraction parameter $\beta$ which in this case 
	depends on the smoothing parameter $s$. Indeed, from Lemma \ref{L:StronglyConvexProxContraction} 
	(see Appendix \ref{A:ProofStrongMoreau}), we have that
	\begin{equation*}
		\beta = \frac{1}{1 + s\sigma} .
	\end{equation*}
	Therefore, we suggest the following concept of rate of convergence result which is different than 
	the classical one, but seems to be relevant when discussing algorithms for solving bi-level 
	problems.
\medskip

	Let $\delta > 0$ be the required uniform accuracy in terms of the outer objective function, that 
	is,
	\begin{equation} \label{OmegaAccuracy}
		\omega\left(\bx^{k}\right) - M_{s\omega}\left(\bx^{k}\right) \leq \delta, \quad \forall \, k 
		\in \nn,
	\end{equation}
	where it should be remembered that $\omega\left(\bx^{k}\right) - M_{s\omega}\left(\bx^{k}\right) 
	\geq 0$ for all $k \in \nn$. Now, we would like do determine the number of iterations $K$ that is
	needed to achieve $\varepsilon$-optimal solution of the inner problem, that is,
	\begin{equation*}
		\varphi\left(\bx^{K}\right) - \varphi\left(\bx_{mn}^{\ast}\right) \leq \varepsilon,
	\end{equation*}	
	while keeping the uniform accuracy as given in \eqref{OmegaAccuracy}. This means that $K$ depends 	
	on both $\varepsilon$ and $\delta$.
	\begin{proposition}
		Let $\Seq{\bx}{k}$ be a sequence generated by BiG-SAM and suppose that Assumptions
		\ref{A:Composite}, \ref{A:OmegaNonsmooth} and \ref{A:AlphSeq} hold true. In addition, suppose 
		that the smoothing parameter is chosen by
		\begin{equation*}
			s = \frac{2\delta}{\ell_{\omega}^{2}}.
		\end{equation*}
		Let $t \in \left(0 , 1/L_{f}\right]$. Then, \eqref{OmegaAccuracy} holds true and for
		\begin{equation*}
			k \geq \frac{4C_{\bx_{mn}^{\ast}}^{2}}{t\varepsilon}\left(2 + \frac{3\ell_{\omega}^{2}}
			{2\sigma\delta} + \frac{\ell_{\omega}^{4}}{4\sigma^{2}\delta^{2}}\right) - 1,
		\end{equation*}
		it holds that $\varphi\left(\bx^{k}\right) - \varphi\left(\bx_{mn}^{\ast}\right) \leq 
		\varepsilon$.
	\end{proposition}
	\begin{proof}
		Since $\omega$ is $\ell_{\omega}$-Lipschitz continuous (see Assumption \ref{A:OmegaNonsmooth}) 
		it follows that the norms of the subgradients of $\omega$ are bounded from above by 
		$\ell_{\omega}$. Thus, from \cite[Lemma 4.2]{BT2012} it follows, for all $\bx \in \real^{n}$, 
		that
		\begin{equation} \label{Eq:1}
			\omega\left(\bx\right) - \frac{s\ell_{\omega}^{2}}{2}\leq M_{s\omega}\left(\bx\right) \leq 
			\omega\left(\bx\right).
		\end{equation}
		Therefore, for $s = 2\delta/\ell_{\omega}^{2}$, we obtain that
		\begin{equation*} 
			\omega\left(\bx^{k}\right) - M_{s\omega}\left(\bx^{k}\right) \leq \delta, \quad \forall \, 
			k \in \nn.
		\end{equation*}
		Using Theorem \ref{T:FunctionRate} we have that		
		\begin{equation*}
			\varphi\left(\by^{k + 1}\right) - \varphi\left(\bx_{mn}^{\ast}\right) \leq
			\frac{2C_{\bx_{mn}^{\ast}}^{2}\left(J + 2\right)}{\left(k + 1\right)\left(1 - \beta
			\right)t},
		\end{equation*}
		where $J = \lfloor 2/\left(1 - \beta\right) \rfloor$. Substituting $\beta =  1/\left(1 + s
		\sigma\right)$ in the above bound yields that
		\begin{align*}
			\varphi\left(\by^{k + 1}\right) - \varphi\left(\bx_{mn}^{\ast}\right) & \leq
			\frac{2C_{\bx_{mn}^{\ast}}^{2}\left(\frac{2}{1 - \beta} + 2\right)}{\left(k + 1\right)
			\left(1 - \beta\right)t} \\ 
			& = \frac{4C_{\bx_{mn}^{\ast}}^{2}}{\left(k + 1\right)t} \cdot \frac{2 - \beta}{\left(1 - 
			\beta\right)^{2}} \\
			& = \frac{4C_{\bx_{mn}^{\ast}}^{2}}{\left(k + 1\right)t}\left(\frac{\left(1 + s\sigma
			\right)^{2}}{\left(s\sigma\right)^{2}} + \frac{1 + s\sigma}{s\sigma}\right) \\	
			& = \frac{4C_{\bx_{mn}^{\ast}}^{2}}{\left(k + 1\right)t}\left(2+\frac{3}{s\sigma}+\frac{1}
			{(s\sigma)^2}\right).	
		\end{align*}
		Now, we use the smoothing parameter that we found above to obtain that
		\begin{equation*}
			\varphi\left(\by^{k + 1}\right) - \varphi\left(\bx_{mn}^{\ast}\right) \leq
			\frac{4C_{\bx_{mn}^{\ast}}^{2}}{\left(k + 1\right)t}\left(2 + \frac{3\ell_{\omega}^{2}}
			{2\sigma\delta} + \frac{\ell_{\omega}^{4}}{4\sigma^{2}\delta^{2}}\right).
		\end{equation*}
		Thus, given $\varepsilon > 0$, in order to obtain $\varphi\left(\by^{k + 1}\right) - \varphi
		\left(\bx_{mn}^{\ast}\right) \leq \varepsilon$ it remains to find values of $k$ for which
		\begin{equation*}
			\frac{4C_{\bx_{mn}^{\ast}}^{2}}{\left(k + 1\right)t}\left(2 + \frac{3\ell_{\omega}^{2}}
			{2\sigma\delta} + \frac{\ell_{\omega}^{4}}{4\sigma^{2}\delta^{2}}\right)\leq \varepsilon,
		\end{equation*}
		which is equivalent to 
		\begin{equation*}
			k \geq \frac{4C_{\bx_{mn}^{\ast}}^{2}}{t\varepsilon}\left(2 + \frac{3\ell_{\omega}^{2}}
			{2\sigma\delta} + \frac{\ell_{\omega}^{4}}{4\sigma^{2}\delta^{2}}\right) - 1.
		\end{equation*}
		The desired result is obtained by choosing $k$ to be the upper bound just obtained.
	\end{proof}
	Two remarks on the just obtained result.
	\begin{remark}
		\begin{itemize}
			\item[$\rm{(i)}$] The uniform accuracy property mentioned in \eqref{OmegaAccuracy} yields 
				that the limit point $\bx_{s}^{\ast}$ of the sequence generated by BiG-SAM satisfies 
				that
				\begin{equation*}
					\omega\left(\bx_{s}^{\ast}\right) - \omega\left(\bx_{mn}^{\ast}\right) \leq M_{s
					\omega}\left(\bx_{s}^{\ast}\right) + \delta - M_{s\omega}\left(\bx_{mn}^{\ast}
					\right) \leq \delta,
				\end{equation*}
				where the first inequality follows by using the two inequalities given in 
				\eqref{OmegaAccuracy} while the second inequality follows from the fact that $\bx_{s}
				^{\ast}$ is a minimizer of $M_{s\omega}$ over $X^{\ast}$ and obviously $\bx_{mn}
				^{\ast} \in X^{\ast}$. This means that $\delta$ also controls the gap between the 
				wished optimal value of $\omega$, that is, $\omega\left(\bx_{mn}^{\ast}\right)$, and 
				the value of $\omega$ evaluated at the optimal solution of the smoothed problem.
			\item[$\rm{(ii)}$] The number of iterations needed to achieve a desired inner function 
				accuracy $\varepsilon$ is therefore $O(1/\varepsilon\delta^{2})$. Consequently, 
				increasing the uniform accuracy parameter $\delta$ by an order of magnitude results 
				in increment of two orders of magnitude in the number of iterations. For example, by 
				taking $\delta = \sqrt{\varepsilon}$ one will results with a rate of $O(1/
				\varepsilon^{2})$ in terms of the inner objective function values.
		\end{itemize}
	\end{remark}
	
\section{Numerical Experiments} \label{Sec:Numerical Experiments}
	In this section we consider the inverse problems tested in \cite[Section 5.2.2]{BS2014} and 
	present a numerical comparison between the MNG and BiG-SAM methods. Linear inverse problems seeks 
	to reconstruct a vector $\bx \in \real^{n}$ from a set of measurements $\bb \in \real^{m}$ which 
	satisfy the following relation $\bb = \bba\bx + \rho\epsilon$ where $\bba : \real^{n} \rightarrow 
	\real^{m}$ is a given linear mapping, $\epsilon \in \real^{m}$ denotes an unknown noise vector and 
	$\rho > 0$ denotes its magnitude. 
\medskip

	There are several ways to solve linear inverse problems using optimization techniques, but here we 
	will focus on the following bi-level formulation. In this case, the inner objective function is 
	defined by
	\begin{equation*}
		\varphi\left(\bx\right) := \norm{\bba\bx - \bb}^{2} + \delta_{X}\left(\bx\right),
	\end{equation*}
	where $\delta_{X}$ is the indicator function over the non-negative orthant $X = \left\{ \bx \in 
	\real^{n} : \, \bx \geq 0 \right\}$. The outer objective function is given by
	\begin{equation*}
		\omega\left(\bx\right) = \frac{1}{2}\bx^{T}\bbq\bx,
	\end{equation*}
	where $\bbq$ is a certain positive definite matrix.
\medskip

	Following \cite{BS2014} we consider three inverse problems phillips, baart, and foxgood which 
	can be found in the ``regularization tools" website\footnote{see \url{http://www2.imm.dtu.dk/~pcha/Regutools/}}.  	
\medskip
	
	For each of these inverse problems we generated the corresponding $1,000 \times 1,000$ exact 
	linear system $\bba\bx = \bb$ by applying the relevant function ('philips','baart','foxgood'). We 
	then performed $100$ Monte-Carlo simulations by adding normally distributed noise with zero mean 
	to the right-hand side vector $\bb$, using three different choices of standard deviation: $\rho = 
	10^{-1}, 10^{-2}, 10^{-3}$. The matrix $\bbq$ is defined by $\bbq = \bbl\bbl' + \bbi$ where $\bbl$ 
	is generated by the function \texttt{get\_l(1,000,1)} from the ``regularization tools" and 
	approximates the first-derivative operator.
\medskip

	In order to implement the MNG method, we need to compute $\bbq^{-1/2}$ and $\bbq^{1/2}$ before the 
	algorithm starts. However, note that while $\bbq$ may be a sparse matrix, the matrices $\bbq^{-1/
	2}$ and $\bbq^{1/2}$ may not be, even if we use other decompositions, such as the Cholesky 
	decomposition. Since the MNG method requires the starting point to be the optimal solution of the 
	unconstrained minimization of $\omega$, we start both algorithms from the point $\bo$.
\medskip

	We tested BiG-SAM with three different choices of the parameter $\gamma$ which are $0.1, 0.5$ and 
	$1$ (see the discussion before Lemma \ref{L:ConvergenceRate} about the parameter $\gamma$). All 
	experiments were ran on a Unix server with 32 Intel Xeon CPUs E5-2690 @2.9GHz and 250GB RAM, using 
	MATLAB R2016a, with no parallelization.
\medskip

	In Table \ref{tbl:TimeComparison} we present the mean time (out of $100$ runs) until the algorithm 
	(MNG and BiG-SAM) reach the stopping criteria $\left(\varphi\left(\by^{k}\right) - \varphi^{\ast}
	\right)/\varphi^{\ast} < 10^{-2}$, where $\varphi^{\ast}$ is the optimal value of the inner problem. 
	If this stopping criteria was not achieved, then we stopped the algorithm after $500$ seconds. In 
	Table \ref{tbl:PerformanceComparison} we present the mean relative feasibility gap (RFG) given by 
	$\Delta \varphi = \left(\varphi\left(\by^{k}\right) - \varphi^{\ast}\right)/\varphi^{\ast}$ and the 
	mean relative optimality gap (ROG) given by $\Delta \omega = \left|\omega\left(\by^{k}\right) - 
	\omega^{\ast}\right|/\omega^{\ast}$, for each algorithm after a running time of $250$ seconds.		
\medskip

	The values of $\varphi^{\ast}$ and $\omega^{\ast}$ were calculated in advance using CVX \cite{cvx} 
	for MATLAB and GUROBI version 7.0.1 solver, which we will refer to hereafter as the standard solver. 
	The value of $\varphi^{\ast}$ was computed as the optimal solution of the inner problem. The value 
	$\omega^{\ast}$ is a lower bound on the optimal value of the outer problem, obtained by solving the 
	following convex problem 
	\begin{equation*}
		\min \left\{ \omega\left(\bx\right) : \, \bx \geq 0, \, \varphi\left(\bx\right) \leq 
		\varphi^{\ast}(1 + \mu) \right\},
	\end{equation*}
	where $\mu$ is a small number for which the problem was solvable (we used $10^{-4}$).
\vspace{0.2in}

	\begin{table}[h]
		\renewcommand{\arraystretch}{1}
		\centering		
		\begin{tabular}{ll|llll}
			\hline
			Problem & $\rho$ & \multicolumn{4}{c}{Mean time (Number of realization terminated at time 
			limit)} \\
			\cline{3-6} 
			&& \multicolumn{3}{c}{BiG-SAM} & MNG \\
			\cline{3-5}
			&& $\gamma = 0.1$ & $\gamma = 0.5$ & $\gamma = 1$ & \\
			\hline
			\hline
			\multirow{3}{*}{Baart}    & $10^{-1}$ & {\bf 5.37e$-$3} (0) &	3.62e$-$2 (0)   & 6.08e$-$2   
			(0) & 2.92e$-$1 (0)\\
			& $10^{-2}$ & {\bf 1.51e$-$1} (0)&	5.03e$-$1  (0)  & 8.26e$-$1   (0)      & 4.40   (0)\\
			& $10^{-3}$ & {\bf 9.78} (0)&	2.23e$+$1   (0) & 3.57e$+$1  (0)  & 4.18e$+$2 (31)\\
			\hline
			\multirow{3}{*}{Foxgood}  & $10^{-1}$ & {\bf 1.51e$-$2} (0)& 6.88e$-$2  (0)  & 1.06e$-$1	 
			(0)  & 3.33e$-$1 (0)\\
			& $10^{-2}$ & {\bf 4.47e$-$1} (0)&	1.20    (0)     & 2.17  (0)       & 3.65 (0)\\
			& $10^{-3}$ & {\bf 1.30e$+$1} (1)&	2.99e$+$1   (0) & 4.43e$+$1	(1)   & 2.93e$+$1 (1)\\
			\hline
			\multirow{3}{*}{Phillips} & $10^{-1}$ & {\bf 1.13e$-$2} (0)&	3.90e$-$2   (0) & 6.58e$-$2  
			(0)  & 4.02e$-$1 (0)\\
			& $10^{-2}$ & {\bf 2.44}   (0)   &	6.77 (0)    & 9.83 (0)   & 1.67e$+$2 (5)\\
			& $10^{-3}$ & {\bf 4.93e$+$2}    (97)& 4.98e$+$2  (98)& 4.99e$+$2  (99)& 5.00e$+$2  (100)\\
			\hline
		\end{tabular}
		\caption{Comparison between MNG and BiG-SAM ($3$ versions) of mean running times (in seconds) 
		until termination and the number of realizations terminated because of the time limit (of $500$ 
		seconds) over $100$ realization. The comparison is across the different problem instances and 
		noise magnitude $\rho$.}
		\label{tbl:TimeComparison}
	\end{table}	
	 
	The average number of iterations needed to reach these results for the MNG method was usually higher 
	than that of BiG-SAM with $\gamma = 0.1$ but lower than that of BiG-SAM with $\gamma = 0.5$. 
	However, as we can clearly see in the table above the higher iteration cost of this method causes 
	the mean time of the MNG method to be the highest, in most cases, and causes the algorithm to stop 
	because of time limit rather than because it reached the termination criteria.  
\medskip
	
	In Table \ref{tbl:PerformanceComparison} we see that when all methods are ran for the same amount of 
	time, all BiG-SAM variants (except for BiG-SAM with $\gamma = 1$ for the Foxgood with $\rho = 
	0.001$) obtain superior $\Delta\varphi$ values compared to the MNG method, up to $2$ orders of 
	magnitude better. Moreover, in most cases all BiG-SAM variants obtain slightly better $\Delta\omega$ 
	values compared to the MNG method, a fact which is more pronounced for higher values of the noise 
	$\rho$. 
	
	\begin{table}[h]
		\renewcommand{\arraystretch}{1}
		\centering		
		\begin{tabular}{l|c|llll|llll}
			\hline
			Problem & $\rho$ & \multicolumn{4}{c|}{Mean RFG ($\Delta \varphi$ in $\%$)} & 
			\multicolumn{4}{c}{Mean ROG ($\Delta \omega$ in $\%$)}\\
			\cline{3-10} 
			&& \multicolumn{3}{c}{BiG-SAM} & MNG &\multicolumn{3}{c}{BiG-SAM} & MNG \\
			\cline{3-5}\cline{7-9}
			&& $\gamma = 0.1$ & $\gamma = 0.5$ & $\gamma = 1$ & & $\gamma = 0.1$ & $\gamma = 0.5$ & $
			\gamma = 1$ &\\
			\hline
			\hline
			\multirow{3}{*}{Baart}    & $10^{-1}$ & {\bf 6.93e-4} & 1.31e-3 &	1.80e-3 & 3.11e-2 & 26.12 
			&	{\bf 23.87} &	24.88 &	58.08 \\
			& $10^{-2}$ & {\bf 4.61e-2} & 5.42e-2& 5.64e-2 & 1.37e-1 & {\bf 58.13} &58.83 &59.05 & 
			61.16\\
			& $10^{-3}$ & {\bf 1.475e-1} & 1.73e-1 & 2.05e-1 & 3.80 & {\bf 53.26} & 53.40 & 53.50& 
			54.94\\
			\hline
			\multirow{3}{*}{Foxgood}  & $10^{-1}$ & {\bf 1.11e-2} &	1.52e-2 &	1.74e-2 & 5.19e-2 & 
			{\bf 48.37} &	51.13 &	52.51 &	59.98\\
			& $10^{-2}$ & {\bf 3.01e-2} & 3.47e-2 &3.82e-2 & 5.40e-2 & 15.42 & 15.29 & {\bf 15.25} & 
			15.28\\
			& $10^{-3}$ & {\bf 1.88e-2}& 2.46e-2 & 3.55e-2 & 2.56e-2 & 1.65 &1.50 &	{\bf 1.41}& 1.45\\
			\hline
			\multirow{3}{*}{Phillips} & $10^{-1}$ &  {\bf 3.84e-2} &	5.11e-2 &	6.00e-2 & 2.40e-1 & 
			{\bf 78.95} & 81.73 & 83.14 & 90.03 \\
			& $10^{-2}$ & {\bf 4.51e-1} &4.98e-1 &5.26e-1 &7.44e-1 & {\bf 93.91}&94.01& 94.06& 94.16\\
			& $10^{-3}$ & {\bf 1.75} & 1.82 & 1.87 &	 2.19 &  {\bf 90.16} & 90.16 & 90.16 & 90.17\\
			\hline
		\end{tabular}
		\caption{Comparison of relative feasible gap (RFG) and relative optimal gap (ROG) after $250$ 
		seconds for MNG and BiG-SAM with various parameters, averaged over $100$ realization for each 
		instance of problem and noise magnitude $\rho$.}
		\label{tbl:PerformanceComparison}
	\end{table}	

	In order to better understand this comparison we look at a specific realization of size $100$ for a 
	problem of Phillips type with $\rho = 0.01$. In Figure \ref{fig:f_values_vs._Time} we can see that 
	BiG-SAM (with $\gamma = 1$) and MNG are very close in the first $10$ seconds, but then BiG-SAM 
	starts to improve much faster than the MNG method. Moreover, we can see clearly that lower value of 
	$\gamma$ yields a faster convergence. In Figure \ref{fig:norm_x_xstar_vs_Time} we see the distance 
	between the iteration $\bx^{k}$ and the optimal solution $\bx^{\ast}$ (which was evaluated via the 
	same procedure we used to find $\omega^{\ast}$). We can see the same behavior here as in the first 
	figure, which means, a faster convergence of the BiG-SAM variants. 

	\begin{figure}
		\begin{subfigure}{.5\textwidth}
			\centering
			\includegraphics[width=1\linewidth]{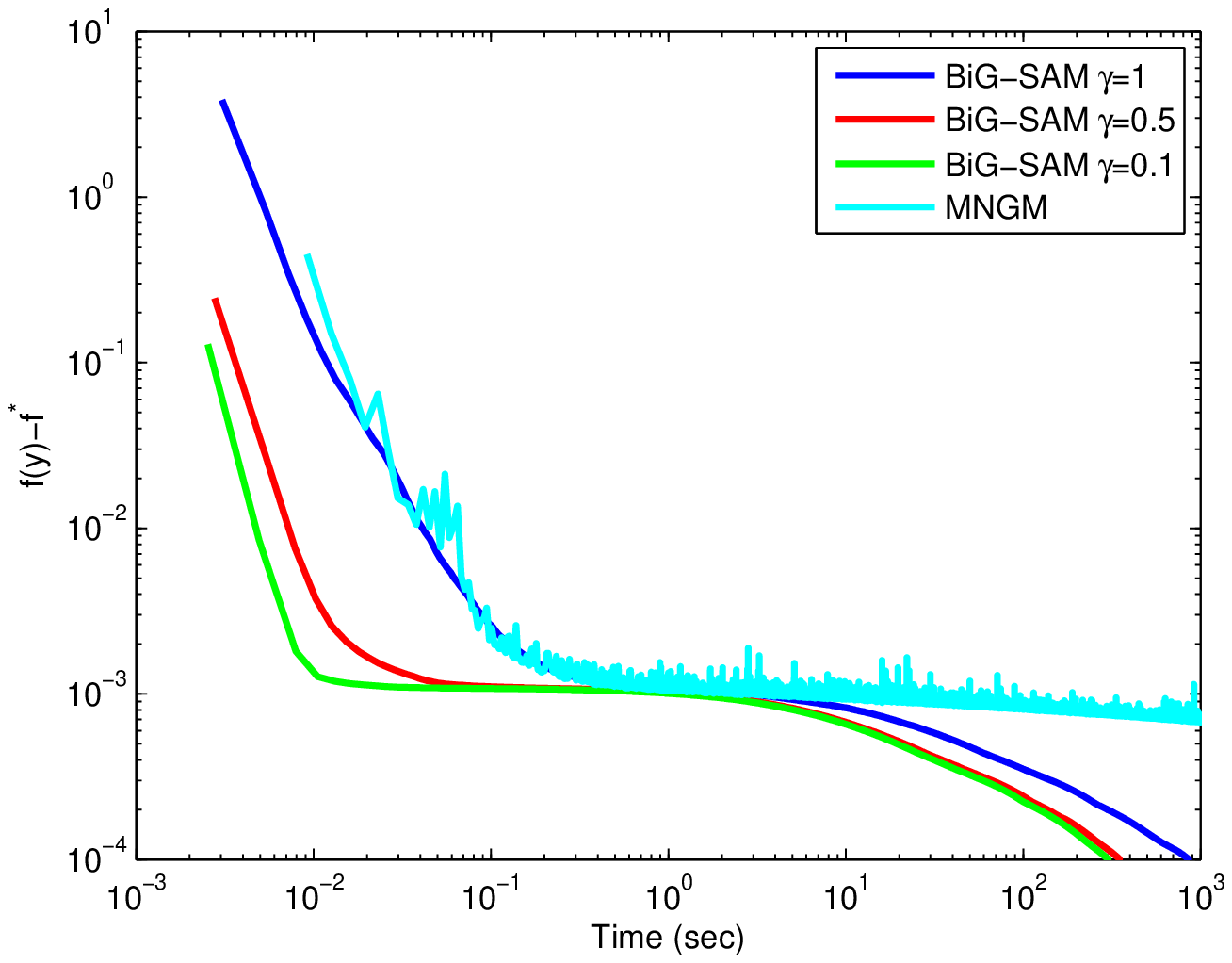}
			\caption{Value of the inner objective function vs. time}
			\label{fig:f_values_vs._Time}
		\end{subfigure}
		\begin{subfigure}{.5\textwidth}
			\centering
			\includegraphics[width=1\linewidth]{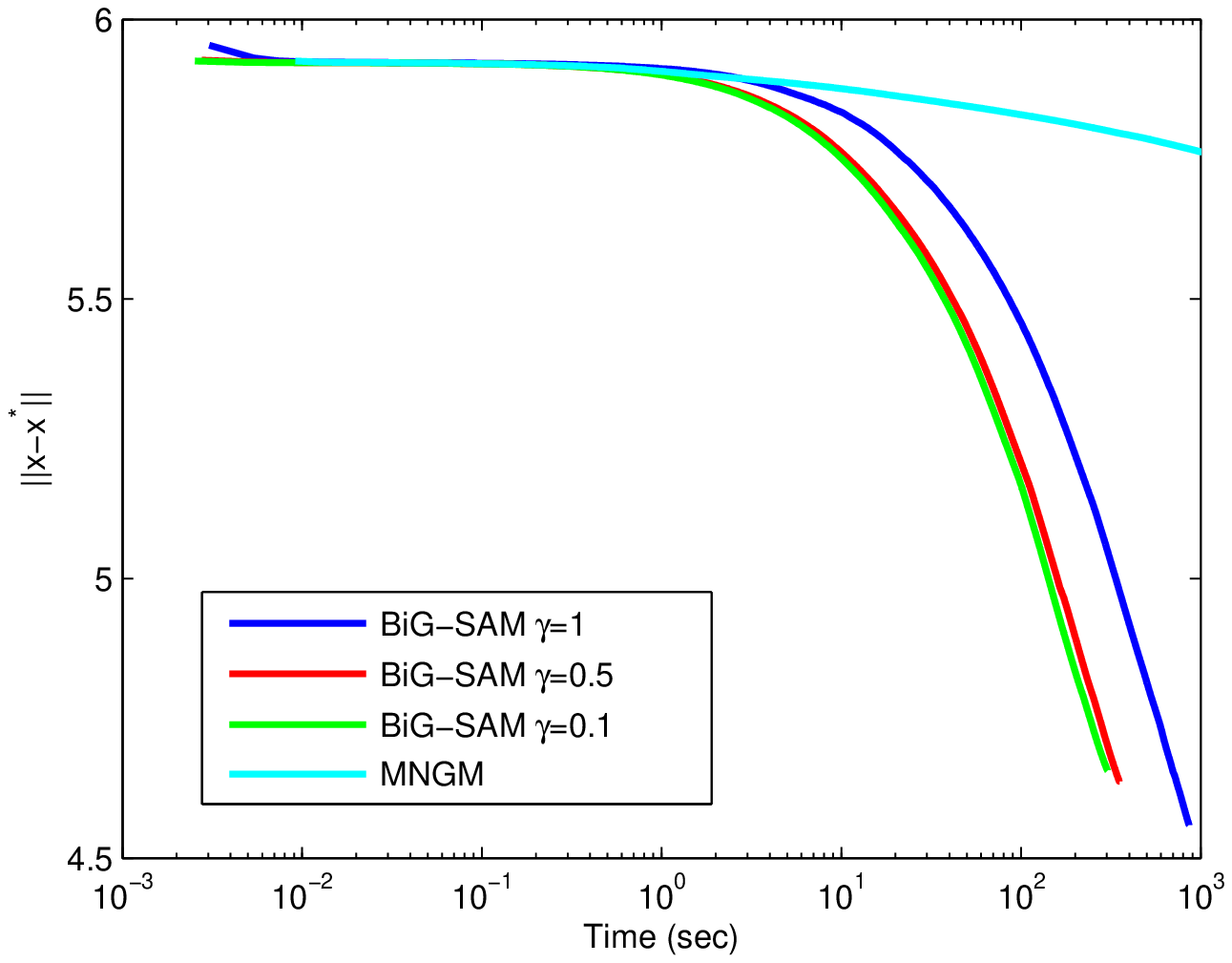}
			\caption{Distance to optimal solution vs. time}
			\label{fig:norm_x_xstar_vs_Time}
		\end{subfigure}
		\caption{The progress of the algorithms in time for a Phillips example with $\rho = 0.01$ and 
		$n = 100$.}
	\end{figure}

\newpage
\begin{appendices}

\section{Proof of Proposition \ref{P:StrongConvexMoreauEnvelope}} \label{A:ProofStrongMoreau}
	We provide two proofs. One is short and based on well-known facts of convex analysis. The second 
	proof is more complicated but provide useful properties of the mappings that play a central role 
	in this work and will be useful in other contexts. We begin with first proof.
	\begin{proof}
		The proof is simple and based on the notion of infimal convolution and its properties. First 
		of all, by its definition we have that the Moreau envelope is the infimal convolution of 
		$\omega$ with the the squared norm function $h\left(\cdot\right) = \left(1/2s\right)
		\norm{\cdot}^{2}$. Thus, for all $\bx \in \real^{n}$, we have that $M_{s\omega}\left(\bx
		\right) = \left(\omega \Box h\right)\left(\bx\right)$. A well-known fact (see 
		\cite[Proposition 13.21(i), Page 187]{BC2011-B}) yields that $\left(\omega \Box h
		\right)^{\ast} = \omega^{\ast} + h^{\ast}$, and therefore $M_{s\omega}^{\ast} = \omega^{\ast} 
		+ h^{\ast}$. Now, the rest of the proof follows from the following known fact: function 
		$\vartheta : \real^{n} \rightarrow \left(-\infty , \infty\right]$ is strongly convex with 
		strong convexity parameter $t$ if and only if its conjugate $\vartheta^{\ast}$ is a 
		continuously differentiable function whose gradient is Lipschitz continuous with constant $1/t
		$. Using this fact twice yields that $M_{s\omega}^{\ast}$ has Lipschitz continuous gradient 
		with constant $s + 1/\sigma$. Now, using the converse implication of the fact gives the 
		desired result.
	\end{proof}
	Before we give the second proof, we will prove that the proximal mapping of a strongly convex 
	function is a $\beta$-contraction, where $\beta = 1/\left(1 + s\sigma\right)$.
	\begin{lemma} \label{L:StronglyConvexProxContraction}
		Let $\omega : \real^{n} \rightarrow \left(-\infty , +\infty\right]$ be a strongly convex 
		function with parameter $\sigma$. Then, for any $s > 0$, it follows $\prox_{s\omega}$ is a 
		$\beta$-contraction, where $\beta = 1/\left(1 + s\sigma\right)$, that is,
		\begin{equation*}
			\norm{\prox_{s\omega}\left(\bx\right) - \prox_{s\omega}\left(\by\right)} \leq \frac{1}{1 + 
			s\sigma}\norm{\bx - \by}, \quad \forall \,\, \bx , \by \in \real^{n}.
		\end{equation*}
	\end{lemma}
	\begin{proof}
		We first define an auxiliary function $\phi$ by $\phi\left(\bx\right) = s\omega\left(\bx
		\right) - \left(s\sigma/2\right)\norm{\bx}^{2}$. Since $\omega$ is strongly convex with 
		parameter $\sigma$ it follows that $\varphi$ is convex. Hence, by the definition of the 
		proximal mapping, we obtain
		\begin{align}
			\prox_{s\omega}\left(\bx\right) & = \argmin_{\bu \in \real^{n}} \left\{ s\omega\left(\bu
			\right) + \frac{1}{2}\norm{\bu - \bx}^{2} \right\} \nonumber \\
			& = \argmin_{\bu \in \real^{n}} \left\{ \phi\left(\bu\right) + \frac{s\sigma}{2}\norm{\bu}
			^{2} + \frac{1}{2}\norm{\bu - \bx}^{2} \right\} \nonumber \\
			& = \argmin_{\bu \in \real^{n}} \left\{ \phi\left(\bu\right) + \frac{1 + s\sigma}{2}
			\norm{\bu - \frac{1}{1 + s\sigma}\bx}^{2} \right\} \nonumber \\
			& = \argmin_{\bu \in \real^{n}} \left\{ \frac{1}{1 + s\sigma}\phi\left(\bu\right) + 
			\frac{1}{2}\norm{\bu - \frac{1}{1 + s\sigma}\bx}^{2} \right\} \nonumber \\
			& = \prox_{\frac{1}{1 + s\sigma}\phi}\left(\frac{\bx}{1 + s\sigma}\right).
			\label{L:StronglyConvexProxContraction:1}
		\end{align}
		From \eqref{L:StronglyConvexProxContraction:1} and the non-expensiveness of the proximal 
		mapping (see \cite[Proposition 12.27, Page 176]{BC2011-B}) we get
		\begin{align*}
			\norm{\prox_{s\omega}\left(\bx\right) - \prox_{s\omega}\left(\by\right)} & =
			\norm{\prox_{\frac{1}{1 + s\sigma}\phi}\left(\frac{\bx}{1 + s\sigma}\right) -
			\prox_{\frac{1}{1 + s\sigma}\phi}\left(\frac{\by}{1 + s\sigma}\right)} \\
			& \leq \frac{1}{1 + s\sigma}\norm{\bx - \by}.
		\end{align*}
		This proves that the proximal mapping is $1/\left(1 + s\sigma\right)$-contraction.
	\end{proof}
	Now we can provide the second proof of Proposition \ref{P:StrongConvexMoreauEnvelope}.
	\begin{proof}
		By \eqref{MoreauEnvelopeGradient} and using the Cauchy-Schwartz inequality we have that
		\begin{align*}
			\act{\nabla M_{s\omega}\left(\bx\right) - \nabla M_{s\omega}\left(\by\right) , \bx - \by}
			& = \frac{1}{s}\act{\bx - \by - \left(\prox_{s\omega}\left(\bx\right) - \prox_{s\omega}
			\left(\by\right)\right) , \bx - \by} \\
			& = \frac{1}{s}\norm{\bx - \by}^{2} - \frac{1}{s}\act{\prox_{s\omega}\left(\bx\right) - 
			\prox_{s\omega}\left(\by\right) , \bx - \by} \\
			& \geq \frac{1}{s}\norm{\bx - \by}^{2} - \frac{1}{s}\norm{\bx - \by}\cdot\norm{\prox_{s
			\omega}\left(\bx\right) - \prox_{s\omega}\left(\by\right)}.
		\end{align*}
		By Lemma \ref{L:StronglyConvexProxContraction} we have that
		\begin{equation*}
			\norm{\prox_{s\omega}\left(\bx\right) - \prox_{s\omega}\left(\by\right)} \leq \frac{1}{1 + 
			s\sigma}\norm{\bx - \by}.
		\end{equation*}
		Thus combining the two inequalities we obtain that
		\begin{equation*}
			\act{\nabla M_{s\omega}\left(\bx\right) - \nabla M_{s\omega}\left(\by\right) , \bx - \by} 
			\geq \frac{1}{s}\left(1 - \frac{1}{1 + s\sigma}\right)\norm{\bx - \by}^{2} = \frac{\sigma}
			{1 + s\sigma}\norm{\bx - \by}^{2}.
		\end{equation*}
		Thus we conclude that $M_{s\omega}$ is strongly convex with parameter $\sigma/\left(1 + s
		\sigma\right)$.
	\end{proof}

\section{Proof of Proposition \ref{P:GradContraction}} \label{A:ProofGradContraction}
	Denote $\tilde{\bx} := \bx - s\nabla \omega\left(\bx\right)$ and $\tilde{\by} := \by - s\nabla 
	\omega\left(\by\right)$. By the definition of $\tilde{\bx}$ and $\tilde{\by}$ we have that
	\begin{align}
		\norm{\tilde{\bx} - \tilde{\by}}^{2} & = \norm{\bx - s\nabla \omega\left(\bx\right) - 
		\left(\by - s\nabla \omega\left(\by\right)\right)}^{2} \nonumber \\
		& = \norm{\bx - \by}^{2} - 2s\act{\nabla \omega\left(\bx\right) - \nabla \omega\left(\by
		\right) , \bx - \by} + s^{2}\norm{\nabla \omega\left(\bx\right) - \nabla \omega\left(\by
		\right)}^{2}. 
		\label{P:GradContraction:1}
	\end{align}
	Since $\omega$ is $\sigma$-strongly convex then by \cite[Theorem 2.1.12, Page 66]{N04} we have 
	that
	\begin{equation} \label{P:GradContraction:2}
		\act{\nabla \omega\left(\bx\right) - \nabla \omega\left(\by\right) , \bx - \by} \geq 
		\frac{\sigma L_{\omega}}{\sigma + L_{\omega}}\norm{\bx - \by}^{2} + \frac{1}{\sigma + 
		L_{\omega}}\norm{\nabla \omega\left(\bx\right) - \nabla \omega\left(\by\right)}^{2}.
	\end{equation}
	By combining \eqref{P:GradContraction:1} and \eqref{P:GradContraction:2} we obtain that
	\begin{equation*}
		\norm{\tilde{\bx} - \tilde{\by}}^{2} \leq \left(1 - \frac{2s\sigma L_{\omega}}{\sigma + 
		L_{\omega}}\right)\norm{\bx - \by}^{2} + \left(s^{2} - \frac{2s}{\sigma + L_{\omega}}\right)
		\norm{\nabla \omega\left(\bx\right) - \nabla \omega\left(\by\right)}^{2}.
	\end{equation*}
	Therefore, for any $s \leq 2/\left(\sigma + L_{\omega}\right)$, we have that the second term is 
	negative and so
	\begin{equation*}
		\norm{\bx - s\nabla \omega\left(\bx\right) - \left(\by - s\nabla \omega\left(\by\right)
		\right)} \leq \sqrt{1 - \frac{2s\sigma L_{\omega}}{\sigma + L_{\omega}}}\norm{\bx - \by},
	\end{equation*}
	which proves the desired result. \qed

\section{Proof of Lemma \ref{L:Alpha}} \label{A:ProofLemmaAlpha}
	We split the proof into two cases: $k \leq J$ and $k > J$. We will start with proving the desired 
	result for $k \leq J$.
\medskip

	\underline{{\bf Case 1:}}
\smallskip
		
	Since $J \geq 2$ and $\gamma \leq 1$ it follows that $a_{1} \leq M < 2M \leq MJ/\gamma$, and since 
	$b_{k} = 1$ for any $k \leq J$ we have that
	\begin{equation*}
		a_{k} = \left(1 - \gamma\right)a_{k - 1}, \quad k = 2 , 3 , \ldots , J.
	\end{equation*}
	Now, using the fact that $\gamma k \leq J$, we obtain 
	\begin{equation*}
		a_{k} = \left(1 - \gamma\right)^{k -  1}a_{1} \leq a_{1} \leq \frac{Ja_{1}}{\gamma k} \leq 
		\frac{MJ}{\gamma k}, \quad k = 2 , 3 , \ldots , J.
	\end{equation*}
	This proves that the desired result hols true for all $k \leq J$.
\medskip

	\underline{{\bf Case 2:}}
\smallskip			

	We will assume that the claim is true for all $l = 1 , 2 , \ldots , k$ where $k \geq J$ and prove 
	that it is true for $k + 1$. In this case, it is clear that $b_{k + 1} = 2/\left(\gamma\left(k + 
	1\right)\right)$ and $b_{k} \leq 2/\left(\gamma k\right)$. Since $c_{k} \leq M$ and using the 
	induction assumption we obtain
	\begin{align*}
		a_{k + 1} & \leq\left(1 - \gamma b_{k + 1}\right)a_{k} + \left(b_{k} - b_{k + 1}\right) c_{k} 
		\\
		& \leq \left(1 - \frac{2\gamma}{\gamma\left(k + 1\right)}\right)\frac{JM}{\gamma k} + 
		\left(b_{k} - \frac{2}{\gamma\left(k + 1\right)}\right)M \\
		& \leq \left(1 - \frac{2}{k + 1}\right)\frac{JM}{\gamma k} + \left(\frac{2}{k} - \frac{2}{k + 
		1}\right)\frac{M}{\gamma} \\
		& = \frac{k - 1}{k + 1} \cdot \frac{JM}{\gamma k} + \frac{2M}{\gamma k\left(k + 1\right)} \\
		& \leq\frac{JM\left(k - 1\right)}{\gamma k\left(k + 1\right)} + \frac{JM}{\gamma k\left(k + 
		1\right)} \\
		& = \frac{JM}{\gamma\left(k + 1\right)},
	\end{align*}
	where the last inequality follows from the fact that $k > J \geq 2$. Thus the claim for $k > J$ is 
	also proven. This completes the proof of the desired result. \qed
\end{appendices}
		
\bibliographystyle{plain}
\bibliography{notes-1}
\end{document}